\documentclass{amsart}
\author{Julien Melleray}
\author{Simon Robert}
\address{Universit\'e Claude Bernard -- Lyon 1 \\
  Institut Camille Jordan, CNRS UMR 5208 \\
  43 boulevard du 11 novembre 1918 \\
  69622 Villeurbanne Cedex \\
  France}
\usepackage{amssymb}
\usepackage[initials,shortalphabetic]{amsrefs}
\usepackage{url}
\usepackage[all]{xy}
\usepackage[autostyle]{csquotes}
\usepackage{newpxtext,newpxmath}
\usepackage{hyperref}
\usepackage{my-macros}
\usepackage{graphicx}
\numberwithin{equation}{section}

\title[From invariant measures to orbit equivalence]{From invariant measures to orbit equivalence, via locally finite groups}

\newcounter{dummy}
\makeatletter
\newcommand\myitem[1][]{\item[#1]\refstepcounter{dummy}\def\@currentlabel{#1}}
\makeatother
\begin{document}

\begin{abstract}
We give a new proof of a theorem of Giordano, Putnam and Skau characterizing orbit equivalence of minimal homeomorphisms of the Cantor space in terms of their sets of invariant Borel probability measures. The proof is based on a strengthening of a theorem of Krieger concerning minimal actions of certain locally finite groups of homeomorphisms, and we also give a new proof of the Giordano--Putnam--Skau characterization of orbit equivalence for these actions.
\end{abstract}
\maketitle

\section{Introduction}
This paper is concerned with continuous actions of countable groups by homeomorphisms on the Cantor space $X$; we say that such an action is \emph{minimal} if all of its orbits are dense. Two actions of countable groups $\Gamma,\Lambda$ on $X$ are \emph{orbit equivalent} if there exists a homeomorphism $g$ of $X$ such that 
\[\forall x,x' \in X \quad \left( \exists \gamma \in \Gamma \ \gamma(x)=x' \right) \Leftrightarrow \left( \exists \lambda \in \Lambda \ \lambda(g(x))= g(x') \right) \]
In words, the actions of $\Gamma$ and $\Lambda$ are orbit equivalent when the equivalence relations that they induce are isomorphic. In what follows, we will often see $\Gamma$, $\Lambda$ as subgroups of $\Homeo(X)$, and then we simply say that $\Gamma$, $\Lambda$ are orbit equivalent.

For any minimal homeomorphism $\varphi$ of $X$ (i.e.~a homeomorphism such that the associated action of $\Z$ on $X$ is minimal), denote by $M(g)$ the set of all $g$-invariant Borel probability measures. Our motivation here is to give a proof of the following famous theorem of Giordano, Putnam and Skau, which we call throughout the paper the \enquote{classification theorem for minimal homeomorphisms}.

\begin{theorem*}[Giordano--Putnam--Skau \cite{Giordano1995}*{Theorem~2.2}]
Assume that $\varphi, \psi$ are minimal homeomorphisms of the Cantor space $X$. The following conditions are equivalent.
\begin{itemize}
\item[•] The $\Z$-actions induced by $\varphi$ and $\psi$ are orbit equivalent.
\item[•] There exists a homeomorphism $g$ of $X$ such that $g_*M(\varphi)=M(\psi)$.
\noindent (where $g_*M(\varphi)= \{g_* \mu \colon \mu \in M(\psi)\}$, with $g_* \mu(A)= \mu(g^{-1}(A))$ for any Borel subset $A$ of $X$)
\end{itemize}
\end{theorem*}

The implication from top to bottom above is easy to see, and any $g$ witnessing that $\varphi$ and $\psi$ are orbit equivalent is such that $g_ * M(\varphi)=M(\psi)$. The converse is much more surprising, and somewhat mysterious; in particular, two minimal homeomorphisms may have the same invariant Borel probability measures but induce different equivalence relations on $X$. The original proof in \cite{Giordano1995} uses homological algebra in an essential way; different arguments have been proposed since then (see for instance \cite{Putnam2010} or \cite{Putnam2018}, as well as the \enquote{elementary} proof given in \cite{Keane2011}) but it remains difficult to \enquote{understand the
dynamics that lie beneath}, to quote Glasner and Weiss \cite{Glasner1995a}. 

Our aim in this paper is to present a self-contained proof of the classification theorem for minimal $\Z$-actions. We exploit a connection with minimal actions of certain locally finite groups, which was already noticed by Giordano, Putnam and Skau. Recall that a subgroup $\Gamma$ of $\Homeo(X)$ is a \emph{full group} if 
\[\Gamma = \{g \in \Homeo(X) \colon \exists A \text{ finite } \subset \Gamma \ \forall x \in X \ \exists \gamma \in A \ \gamma(x)=x\} \]
The condition above says that $\Gamma$ coincides with the group of homeomorphisms which are obtained by gluing together finitely many elements of $\Gamma$ (see section \ref{s:fullgroups} for some more details on full groups, as well as more background on other notions that we use in this introduction. In the context of minimal $\Z$-actions on the Cantor space, full groups were first investigated in \cite{Giordano1999}). Following Krieger \cite{Krieger1979}, we say that a subgroup $\Gamma$ of $\Homeo(X)$ is \emph{ample} if it is a countable, locally finite full group such that $\{x \colon \gamma(x)= x\}$ is clopen for every $\gamma \in \Gamma$. Equivalence relations induced by minimal ample groups correspond to the \enquote{affable} equivalence relations introduced by Giordano, Putnam and Skau, which play an important part in their theory (see \cite{Giordano2004} and \cite{Putnam2018}). For any ample group $\Gamma$, the set $M(\Gamma)$ of all $\Gamma$-invariant Borel probability measures on $X$ is nonempty, and Giordano, Putnam and Skau established the following result, which we call the \enquote{classification theorem for minimal ample groups}  (see also Putnam \cite{Putnam2010} for an interesting take on this result).
\begin{theorem*}[Giordano--Putnam--Skau \cite{Giordano1995}*{Theorem~2.3}]
Let $X$ be the Cantor space. Given two ample subgroups of $\Homeo(X)$ acting minimally, the following conditions are equivalent.
\begin{itemize}
\item[•] The actions of $\Gamma$ and $\Lambda$ are orbit equivalent.
\item[•] There exists a homeomorphism $g$ of $X$ such that $g_*M(\Gamma)=M(\Lambda)$.
\end{itemize}
\end{theorem*}
Giordano, Putnam and Skau observed that every minimal $\Z$-action is orbit equivalent to an action of an ample group (the group in question is naturally associated to the $\Z$-action, see section \ref{s:Kakutani}) and then pointed out that one can deduce the classification theorem for minimal $\Z$-actions from the corresponding classification theorem for actions of ample groups. Our strategy is the same, although our proof of this fact (Theorem \ref{t:classification_minimal_homeomorphisms} here) is new. Like Giordano, Putnam and Skau we use an ``absorption theorem'' (Theorem \ref{t:absorption}) to establish the classification theorem for minimal ample groups.

As an example of such an absorption theorem (indeed, a fundamental particular case), consider the relation $E_0$ on $\{0,1\}^\N$, were for all $x,y \in \{0,1\}^\N$ one has $x \, E_0 \, y $ iff $x(n)=y(n)$ for all large enough $n$. This relation is induced by the countable, locally finite group $\Gamma$ of all dyadic permutations; these are the homeomorphisms of $\{0,1\}^\N$ obtained by first choosing a bijection $\sigma$ of some $\{0,1\}^n$ to itself, then setting $\tilde \sigma(u \frown x)= \sigma(u) \frown x$. 
Denote by $\varphi$ the dyadic odometer (which corresponds to ``adding $1$ with right-carry''), which is a minimal homeomorphism; the relation $R_\varphi$ is obtained from $E_0$ by gluing the classes of $0^\infty$ and $1^\infty$ together (indeed, $\varphi$ can only change finitely many coordinates at a time, except for the special case $\varphi(1^\infty)=0^\infty$). Giordano, Putnam and Skau's absorption theorem implies in particular that $R_\Gamma$ and $R_\varphi$ are isomorphic; this means that there exists a homeomorphism $g$ of $\{0,1\}^\N$ which turns $E_0$ into the relation obtained from $E_0$ by gluing two orbits together. This extraordinary fact is quite hard to visualize, and we do not know of any ``concrete'' construction of such a $g$. Clemens, in unpublished work \cite{Clemens2008}, gave a nice explicit construction of a minimal homeomorphism inducing $E_0$ (via a Bratteli diagram) and asked whether $E_0$ can be generated by a Lipschitz automorphism of $\{0,1\}^\N$. 

While our absorption theorem (and its strengthening, Theorem \ref{t:absorption2})  follows from known results (\cite{Giordano2004}*{Lemma~4.15}, itself generalized in \cite{Giordano2008}*{Theorem~4.6} and \cite{Matui2008}*{Theorem~3.2}), its proof is relatively elementary, and based on the following strengthening of a theorem of Krieger which we find interesting in its own right.

\begin{thm:Krieger}[see Krieger \cite{Krieger1979}*{Theorem~3.5}]
Let $\Gamma$, $\Lambda$ be two ample subgroups of $\Homeo(X)$; assume that for any $A, B \in \Clopen(X)$ there exists $\gamma \in \Gamma$ such that $\gamma(A)=B$ iff there exists $\lambda \in \Lambda$ such that $\lambda(A)=B$. 

Assume additionally that $K$ is a closed subset of $X$ which intersects each $\Gamma$-orbit in at most one point; $L$ is a closed subset of $X$ which intersects each $\Lambda$-orbit in at most one point; and $h \colon K \to L$ is a homeomorphism. 

Then there exists $g \in \Homeo(X)$ such that $g \Gamma g^{-1}=\Lambda$, and $g_{|K}=h$. 
\end{thm:Krieger}

In terms of the Polish topology of $\Homeo(X)$, the original statement of Krieger's theorem (i.e.~with $K=L=\emptyset$ above) can be seen as saying that two ample groups $\Gamma, \Lambda$ are conjugate in $\Homeo(X)$ as soon as their closures are. A lemma due to Glasner and Weiss (Lemma \ref{l:original_Glasner_Weiss} below), which is essential to our approach, implies that, whenever $\varphi$ is a minimal homeomorphism, the closure of its full group is equal to the group of all homeomorphisms which preserve every measure in $M(\varphi)$. Also, two homeomorphisms $\varphi$ and $\psi$ are orbit equivalent if and only if their full groups are conjugate in $\Homeo(X)$. Thus one can see the classification theorem for minimal $\Z$-actions as the statement that the full groups of two minimal homeomorphisms are conjugate as soon as their closures coincide, which suggests a relationship with Krieger's theorem. That is the motivation for our approach.

Say that an ample subgroup $\Gamma$ of $\Homeo(X)$ acting minimally is \emph{saturated} if 
\[\overline{\Gamma}= \{g \in \Homeo(X) \colon \forall \mu \in M(\Gamma) \ g_* \mu= \mu\}\] 
It follows from Krieger's theorem and a variation on the aforementioned result of Glasner and Weiss (Lemma \ref{l:ample_Glasner_Weiss} below) that two saturated ample groups $\Gamma, \Lambda$ are conjugate iff there exists $g \in \Homeo(X)$ such that $g_*M(\Gamma)=M(\Lambda)$. In particular, a strong form of the classification theorem, in the special case of saturated ample groups, follows directly from a combination of Glasner--Weiss's compactness argument and Krieger's theorem: two saturated minimal ample groups which preserve the same Borel probability measures are conjugate. 

It is then natural to try and prove the classification theorem by establishing that any minimal action of an ample subgroup of $\Homeo(X)$ is orbit equivalent to a minimal action of a saturated ample subgroup (a closely related strategy was suggested in \cite{Ibarlucia2016}). This is the approach that we follow here. 

We need to control some tension between two equivalence relations, one on clopen subsets of $X$ and the other on points. Starting from an ample group $\Gamma$, we want to coarsen the relation induced by the action of $\Gamma$ on $\Clopen(X)$ so as to make $\Gamma$ saturated; but this changes the relation induced by the action of $\Gamma$ on $X$, and the only technique we have at our disposal to control this relation is via our strengthening of Krieger's theorem. Given a minimal ample $\Gamma$, we manipulate Cantor sets $K \sqcup \sigma(K)$, where $\sigma$ is a homeomorphic involution, which intersect each $\Gamma$-orbit in at most one point, and consider the relation $R_{\Gamma,K}$ which is obtained by joining the $\Gamma$-orbit of $x$ and $\sigma(x)$ together for all $x \in K$, and leaving the other orbits unchanged. By Theorem \ref{t:Krieger_pointed}, for any two such $K,L$ the relations $R_{\Gamma,K}$ and $R_{\Gamma,L}$ are isomorphic.
Then we prove that there exists such a $K$ for which $R_{\Gamma,K}$ is induced by an ample group with the same orbits as $\Gamma$ on $\Clopen(X)$ (hence this ample group is conjugate to $\Gamma$); and another such $K$ for which $R_{\Gamma,K}$ is induced by a saturated minimal ample group. It follows that $\Gamma$ is orbit equivalent to a saturated minimal ample group, thereby establishing the classification theorem for minimal ample groups. This is where we need, and prove, an absorption theorem: the argument above establishes that $R_\Gamma$ is isomorphic to the relation obtained by gluing together the orbits of $x$ and $\sigma(x)$ for each $x \in K$.

We conclude this introduction by a few words about the structure of this article, which we tried to make as self-contained as possible, in the hope of making the proof accessible to a broad mathematical audience. Sections \ref{s:background} and \ref{s:ample} mostly consist of background material, the  exception being Theorem \ref{t:Krieger_pointed}, which is the main driving force in our proof of the classification theorems. Section \ref{s:balanced} develops an auxiliary combinatorial tool which we need to extend an ample group to a saturated ample group while having some control on the equivalence relation induced by the bigger group (the underlying idea is related to work from \cite{Ibarlucia2016}). Then we use our tools to prove the classification theorem for minimal ample groups. In the last section, we explain how to deduce the classification theorem for minimal $\Z$-actions.

{\bf Acknowledgements.} The first-named author spent an embarrassing amount of time, spread over several years, trying to understand how to prove the classification theorem for minimal $\Z$-actions via elementary methods. He wishes to thank his co-author for helping to put him out of his misery, as well as present his sincere apologies to the many colleagues whom he has bored with various complaints related to this problem, false proofs, and assorted existential crises (some of which occurred uncomfortably late in this project). 

This article is part of the second-named author's Ph.D. thesis (with the first author as advisor). 

We are grateful to A. Tserunyan and A. Zucker for useful comments and suggestions. Thanks are also due to an anonymous referee for helpful remarks and corrections.

\section{Background and terminology}\label{s:background}
\subsection{Full groups.}\label{s:fullgroups}
\begin{defn}
Given a subgroup $\Gamma$ of $\Homeo(X)$, we set
$$F(\Gamma)=\{g \in \Homeo(X) \colon \exists A \text{ finite } \subset \Gamma \ \forall x \in X \ \exists \gamma \in A  \ g(x)=\gamma(x)\}$$

We say that $F(\Gamma)$ is the topological full group associated to the action of $\Gamma$ on $X$; and that $\Gamma$ is a \emph{full group} if $\Gamma=F(\Gamma)$.
\end{defn}

\begin{nota}
Below, whenever $\Gamma$ is a subgroup of $\Homeo(X)$, we denote by $R_\Gamma$ the equivalence relation induced by the action of $\Gamma$, i.e.~
$$\forall x,y \in X \quad \left( xR_\Gamma y \right) \Leftrightarrow \left( \exists \gamma \in \Gamma \ \gamma (x)=y \right) $$

For $\varphi \in \Homeo(X)$, we simply denote by $R_\varphi$ the equivalence relation induced by the action of $\{\varphi^n \colon n \in \Z\}$.
\end{nota}

\begin{defn}
Let $R$ be an equivalence relation on $X$. The \emph{full group} of $R$ is 
$$[R] = \{g \in \Homeo(X) \colon \forall x \ g(x) R x \}$$
\end{defn}
Note that $[R]$ is indeed a full group as defined above; also, two actions of groups $\Gamma, \Lambda$ on $X$ are orbit equivalent if and only if there exists $g \in \Homeo(X)$ such that $g[R_\Gamma]g^{-1}=[R_\Lambda]$.

When $\varphi$ is a homeomorphism of $X$, we can consider the topological full group of the associated action of $\Z$, which we denote $F(\varphi)$; or the full group of the equivalence relation $R_\varphi$ induced by the action, which we denote $[R_\varphi]$.
The standard notations for these objects are $[[\varphi]]$ for what we denote $F(\varphi)$, and $[\varphi]$ for $[R_\varphi]$; $F(\varphi)$ is called the topological full group of $\varphi$, and $[\varphi]$ the full group of $\varphi$. We find these notations potentially confusing, especially in this paper where it will be important to keep in mind the difference between the full group associated to an action of a countable group, which is a countable group of $\Homeo(X)$, and the full group associated to an equivalence relation, a typically much bigger group. We refer the reader to \cite{Giordano1999}, where topological full groups of minimal homeomorphisms are investigated in detail.

\subsection{Invariant measures for minimal actions}

\begin{nota}\label{nota invariant probability measure}
Whenever $\Gamma$ is a subgroup of $\Homeo(X)$, we denote by $M(\Gamma)$ the set of $\Gamma$-invariant Borel probability measures. For any $\varphi \in \Homeo(X)$, we simply denote $M(\varphi)$ for $M(\{\varphi^n \colon n \in \Z\})$.
\end{nota}

The set $M(\Gamma)$ is nonempty as soon as $\Gamma$ is amenable, which is the case for the groups we are concerned with, namely $\Z$, some locally finite groups (the \emph{ample} groups considered in Section \ref{s:ample}), and full groups associated to these groups and the equivalence relations they induce. 

\begin{lemma}\label{l:invariant_measures_full_group}
Let $\Gamma$ be a countable subgroup of $\Homeo(X)$. Then $M(\Gamma)=M([R_\Gamma])$.
\end{lemma}

\begin{proof}
Since $\Gamma$ is contained in $[R_\Gamma]$, we have $M([R_\Gamma]) \subseteq M(\Gamma)$ by definition. 

Conversely, let $\mu \in M(\Gamma)$ and $g \in [R_\Gamma]$. There exists a Borel partition $(B_n)$ of $X$ and elements $\gamma_n \in \Gamma$ such that $g_{|B_n}={\gamma_n}_{|B_n}$ for all $n$. For any Borel $A$, we have
\begin{align*}
\mu(g (A)) &= \sum_{n \in \N} \mu (g(A \cap B_n)) \\
           &= \sum_{n \in \N} \mu(\gamma_n(A \cap B_n)) \\
           &= \sum_{n \in\N} \mu(A \cap B_n) \\
           &= \mu(A)
\end{align*}
\end{proof}

Given two homeomorphisms $\varphi$, $\psi$, an orbit equivalence from $R_\varphi$ to $R_\psi$ is the same as a homeomorphism $g$ of $X$ such that $g[R_\varphi]g^{-1}=[R_\psi]$. Hence whenever $g$ is an orbit equivalence from $R_\varphi$ to $R_\psi$ we have $g_*M([\varphi])=M([\psi])$, that is, $g_*M(\varphi)=M(\psi)$. This establishes the easy direction of the classification theorem for minimal $\Z$-actions.

We collect some well-known facts about invariant measures for minimal actions.

\begin{lemma}\label{l:props_invariant_measures}
Let $\Gamma$ be a countable subgroup of $\Homeo(X)$ acting minimally; assume that $M(\Gamma) \ne \emptyset$. Then:
\begin{itemize}
    \item Any $\mu \in M(\Gamma)$ is atomless.
    \item For any nonempty clopen $U$, we have $\inf_{\mu \in M(\Gamma)} \mu(U) >0$.
    \item Fix a compatible distance on $X$. For any $\varepsilon >0$, there exists $\delta >0$ such that for any clopen $A$ of diameter less than $\delta$, one has $\sup_{\mu \in M(\Gamma)} \mu(A)< \varepsilon$.
\end{itemize}
\end{lemma}

In particular, any $\Gamma$-invariant measure has full support.
\begin{proof}
Fix $x \in X$ and $\mu \in M(\Gamma)$. The set $\{\gamma (x) \colon \gamma \in \Gamma\}$ is infinite, so $\mu(\{x\})=0$. This proves the first property.

To see why the second property holds, fix a nonempty clopen $U$; since $\Gamma$ acts minimally, we have $\bigcup_{\gamma \in \Gamma} \gamma (U)= X$, so by compactness there exist $\gamma_1,\ldots,\gamma_n \in \Gamma$ with $X= \bigcup_{i=1}^n \gamma_i(U)$, whence $\mu(U) \ge \frac{1}{n}$ for any $\mu \in M(\Gamma)$. 

For the third point, we use the same argument as in (\cite{Bezuglyi-Medynets2008}*{Proposition 2.3}. Arguing by contradiction, we assume that there exists a sequence of clopens $(A_n)$ of vanishing diameter and $\mu_n \in M(\Gamma)$ such that $\mu_n(A_n) \ge \varepsilon$ for all $n$. Up to some extraction, we may assume that $(A_n)$ converges to $x \in X$ for the Vietoris topology on the space of compact subsets of $X$, and $(\mu_n)$ converges to $\mu \in M(\Gamma)$. If $O$ is a clopen neighborhood of $x$, we have that $A_n \subseteq O$ for large enough $n$ so that $\mu_n(O) \ge \mu_n(A_n) \ge \varepsilon$, whence $\mu(O) \ge \varepsilon$ for all $n$. Hence $\mu(\{x\}) \ge \varepsilon$, contradicting the fact that $\mu$ is atomless.
\end{proof}

The following observation will also play a part later on.

\begin{lemma}\label{l:invariant_measures_stabilize}
Let $\Gamma$ be a subgroup of $\Homeo(X)$. Denote
\[G_\Gamma = \lset{g \in \Homeo(X) \colon \forall \mu \in M(\Gamma) \ g_* \mu=\mu} \]
Then $G_\Gamma$ is a full group and $M(G_\Gamma)=M(\Gamma)$.
\end{lemma}

\begin{proof}
Clearly $G_\Gamma$ is a full group; since $\Gamma \subseteq G_\Gamma$ we have $M(G_\Gamma) \subseteq M(\Gamma)$. Conversely, for any $g \in G_\Gamma$ and any $\mu \in M(\Gamma)$ we have $g_* \mu= \mu$, so $M(\Gamma) \subseteq M(G_\Gamma)$.
\end{proof}

\subsection{Kakutani--Rokhlin partitions}\label{s:Kakutani}

In this subsection, we fix a minimal homeomorphism $\varphi$ of $X$. For any nonempty $U \in \Clopen(X)$, we have $X= \bigcup_{n \in \Z} \varphi^n(U)$, thus by compactness of $X$ there exists $N$ such that 
$$X= \bigcup_{n=-N}^N \varphi^n(U)= \bigcup_{n=-2N-1}^{-1} \varphi^n(U)= \bigcup_{n=0}^{2N} \varphi^n(U) $$
This proves that the forward and backward orbit of any $x \in X$ are both dense. For any $x$ there exists some $n \ge 1$ such that $\varphi^n(x) \in U$, and we define 
\[n_U(x)= \min \lset{n \ge 1 \colon \varphi^n(x) \in U} \]
Since $U$ is clopen, the map $n_U$ is continuous, so it takes finitely many values on $U$ since $U$ is compact. Let $I=n_U(X)$ and $U_i=\lset{x \in U \colon n_U(x)=i}$. Each $U_i$ is clopen and for any $i,j \in I$ and any $n \le i-1$, $m \le j-1$ we have $\varphi^n(U_i) \cap \varphi^m(U_j)= \emptyset$ as soon as $(i,n) \ne (j,m)$. This leads us to the following definition.

\begin{defn}
A \emph{Kakutani-Rokhlin}-partition of $X$ is a clopen partition $(U_{i,j})_{i \in I, j < n_i}$ such that for any $i$ and any $j < n_i-1$ one has $\varphi(U_{i,j})=U_{i,j+1}$.

The \emph{base} of the partition is $U=\bigcup_{i \in I} U_{i,0}$, and its \emph{top} is $\varphi^{-1}(U)= \bigcup_{i \in I} U_{i,n_i-1}$. 

We say that $(U_{i,j})_{0\le j < n_i}$ is a \emph{column} of the partition, and that $n_i$ is the \emph{height} of this column. 
\end{defn}

The construction outlined before the definition of Kakutani--Rokhlin partitions shows that for any nonempty clopen $U$ there exists a Kakutani--Rokhlin partition whose base is equal to $U$. 

\begin{defn}
Let $\mcA=(U_{i,j})_{i \in I, 0 \le  j < n_i}$ and $\mcB=(B_{k,l})_{k \in K, 0 \le l < m_k}$ be two Kakutani--Rokhlin partitions.

We say that $\mcB$ refines $\mcA$ if every $B_{k,l}$ is contained in some $A_{i,j}$ and the base of $\mcB$ is contained in the base of $\mcA$ (then the top of $\mcB$ is also contained in the top of $\mcA$). 
\end{defn}

Note that if $\mcB$ refines $\mcA$ and $B_{k,l}$ is contained in some atom $A_{i,j}$ with $j < n_i-1$, then $l < m_k-1$ and $B_{k,l+1} = \varphi(B_{k,l}) \subseteq A_{i,j+1}$; one often says that the columns of $\mcB$ have been obtained by \emph{cutting and stacking} from the columns of $\mcA$. Going back to the example of a Kakutani--Rokhlin partition defined from the first return map to some clopen $U$, the intuition is that if we shrink the base to some $V \subset U$ then $\varphi^k(y)$ can only belong to $V$ if it belongs to $U$; and if $\varphi^k(x) \in U \setminus V$, then before coming back to $U$ one will have to go through the whole column containing $x$ for the Kakutani--Rokhlin partition based on $U$ (see figure \ref{fig: cutting_and_stacking} below). 

\begin{defn}
A Kakutani--Rokhlin partition $\mcA$ is \emph{compatible} with $U \in \Clopen(X)$ if $U$ belongs to the Boolean algebra generated by $\mcA$.
\end{defn}

\begin{lemma}\label{l: making a KR partition compatible with a clopen}
Let $\mcA$ be a Kakutani--Rokhlin partition, and $U \in \Clopen(X)$. There exists a Kakutani--Rokhlin partition $\mcB$ which refines $\mcA$ and is compatible with $U$.
\end{lemma}

\begin{proof}
The proof only involves cutting, and no stacking. Let $\mcA=(U_{i,j})_{i \in I,0 \le j < n_i}$. For any $i \in I$, consider the equivalence relation $R_i$ on $U_{i,0}$ defined by 
$$\left(  x R_i y \right) \Leftrightarrow \left( \forall j < n_i \left( \varphi^j(x) \in U \Leftrightarrow \varphi^j(y) \in U \right) \right)  $$
The equivalence classes of $R_i$ are clopen, call them $D^i_k$, for $k \in K_i$. 

Then the Kakutani--Rokhlin partition $\mcB$ whose columns are $(\varphi^j(D^i_k))_{0 \le j < n_i}$ refines $\mcA$ and is compatible with $U$.
\end{proof}

Note that if $\mcA$ is compatible with some $U \in \Clopen(X)$, and $\mcB$ refines $\mcA$, then $\mcB$ is also compatible with $U$.

Since there are countably many clopen sets, it follows from the previous discussion that, given any $x \in X$, we may build a refining sequence of Kakutani--Rokhlin partitions $(\mcA_n)$ with the following properties:
\begin{enumerate}
    \item \label{KR-part cond 1} For any $U \in \Clopen(X)$, there exists $n$ such that $\mcA_m$ is compatible with $U$ for all $m \ge n$.
    \item \label{KR-part cond 2} The intersection of the bases of $\mcA_n$ is equal to $\{x\}$ (and then the intersection of the tops is equal to $\varphi^{-1}(\{x\})$).
\end{enumerate}

We fix such a sequence of partitions $(\mcA_n)$ for the remainder of this section. 
\begin{remark}\label{KR-part have height as big as we want}
Let $k$ be a given natural integer.
The fact that $\varphi$ acts minimally (aperiodicity of the action would suffice) and condition \eqref{KR-part cond 2} ensure that for any big enough $n$, the base $B$ of $\mcA_n$ is such that $B,\varphi(B),\ldots,\varphi^{k-1}(B)$ are pairwise disjoint, and thus every column of $\mcA_n$ is of height larger than $k$.
\end{remark}

\begin{defn}
For any $n \in \N$, we let $\Gamma_n$ consist of all $g \in \Homeo(X)$ with the following property: for every atom $U_{i,j}$ of $\mcA_n$ there exists an integer $k_{i,j}$ with $0 \le j +k_{i,j}  < n_i$ and such that $g(y)= \varphi^{k_{i,j}}(y)$ for all $y \in U_{i,j}$.

Set $\Gamma_x(\varphi)= \bigcup_n \Gamma_n$ (note that $\Gamma_n$ is a subgroup of $\Gamma_{n+1}$ for all $n$, see figure \ref{fig: cutting_and_stacking} below).
\end{defn}

\begin{figure}[!htb]
    \centering
    \includegraphics[width=\linewidth]{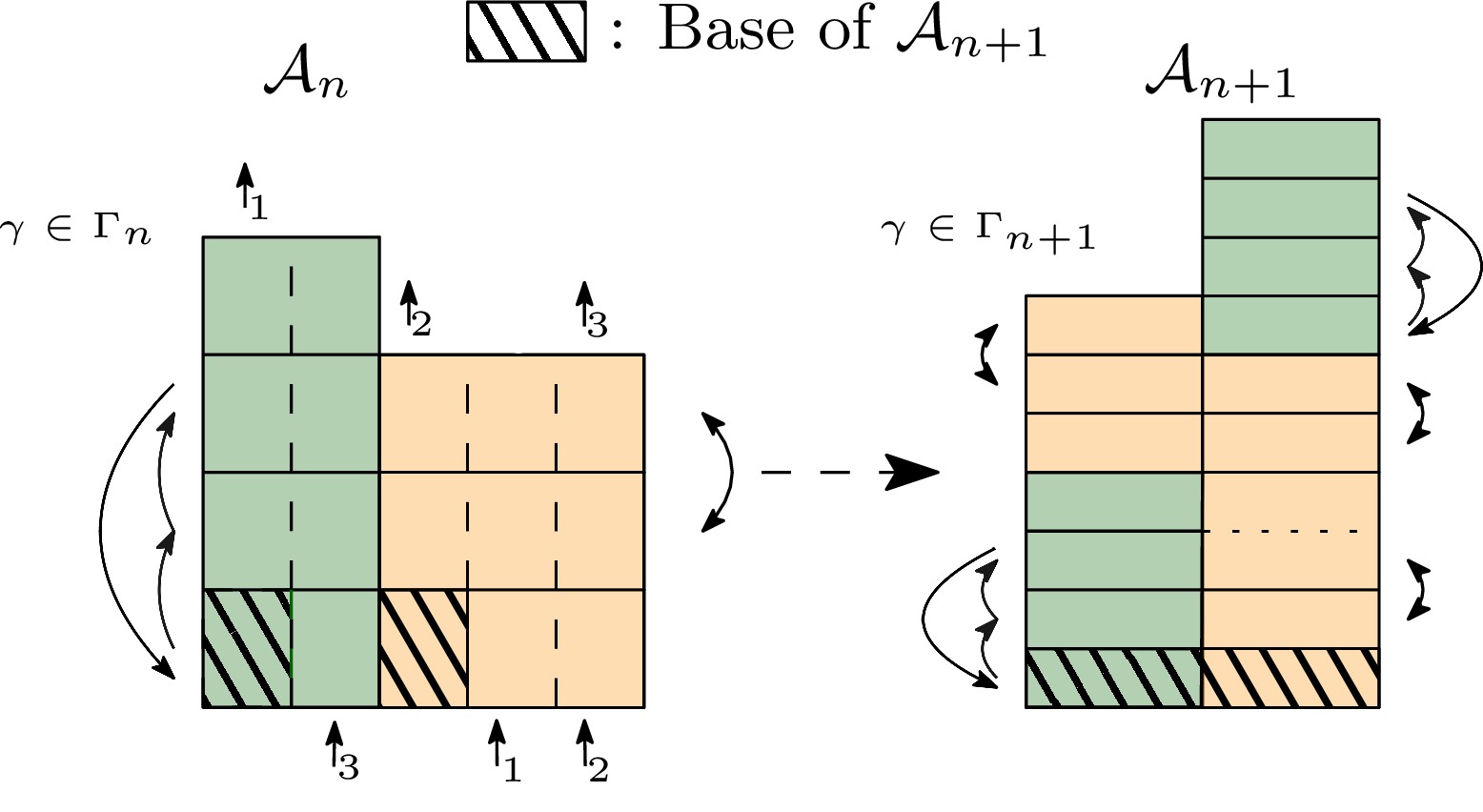}
    \caption{Cutting-and-stacking procedure, and $\Gamma_n < \Gamma_{n+1}$}
    \label{fig: cutting_and_stacking}
\end{figure}


By definition, $\Gamma_x(\varphi)$ is a subgroup of the topological full group $F(\varphi)$.

Each $\gamma \in \Gamma_n$ induces a permutation of the atoms of each column of $\mcA_n$; and if we know how $\gamma$ permutes the atoms within each column then we have completely determined $\gamma$ (this will lead us to the concept of \emph{unit system} in the next section). In particular, each $\Gamma_n$ is finite (and is isomorphic to a direct product of finite symmetric groups). So $\Gamma_x(\varphi)$ is locally finite.

Since each clopen is eventually a union of atoms of $\mcA_n$, $\Gamma_x(\varphi)$ is a full group.

It seems from the definition that the group $\Gamma_x(\varphi)$ depends on the choice of sequence of Kakutani--Rokhlin partitions, but it only depends on the choice of $x$; and even then, we will later see as a consequence of Krieger's theorem that $\Gamma_x(\varphi)$ and $\Gamma_y(\varphi)$ are conjugate in $\Homeo(X)$ for any $x,y \in X$.

\begin{lemma}
Denote $O^+(x)= \lset{\varphi^n(x) \colon n \ge 0}$, $O^-(x)= \lset{\varphi^n(x) \colon n < 0}$. 

Then $\Gamma_x(\varphi)x= O^+(x)$, $\Gamma_x(\varphi)\varphi^{-1}(x)=O^-(x)$, and for any $y$ which is not in the $\varphi$-orbit of $x$ we have $\Gamma_x(\varphi)y= \lset{\varphi^n(y) \colon n \in \Z}$.

Further,
\[\Gamma_x(\varphi)= \lset{g \in F(\varphi) \colon g(O^+(x))=O^+(x)} \]

\end{lemma}

\begin{proof}

Fix $k>0$ and let $n$ be such that each column of $\mcA_n$ has height bigger than $k$. Letting $U_{i,0}$ be the atom of $\mcA_n$ containing $x$, we then have $\varphi^k(x) \in U_{i,k}$, and there is an element of $\Gamma_n$ which is equal to 
$\varphi^k$ on $U_{i,0}$, so $\varphi^k(x) \in \Gamma_x(\varphi)x$. Hence $O^+(x) \subseteq \Gamma_x(\varphi)x$ and a similar argument (or this argument applied to $\varphi^{-1}$) shows that $O^-(x)\subseteq  \Gamma_x(\varphi) \varphi^{-1}(x) $.

The converse inclusions are immediate from the definition of $\Gamma_x(\varphi)$: for $k \ge 0$, and any $n$ such that the height of each column of $\mcA_n$ is bigger than $k$, the $\Gamma_n$-orbit of $\varphi^k(x)$ consists of $x,\ldots,\varphi^n(x)$ (and similarly for negative $k$).

Next, denote by $U_n$ the base of $\mcA_n$. If $y$ does not belong to the $\varphi$-orbit of $x$, then for any sufficiently large $n$ $y$ does not belong to $U_n \cup \varphi^{-1}(U_n)$. If we denote by $V_n$ the atom of $\mcA_n$ which contains $y$, this means that the map which is equal to $\varphi$ on $V_n$ and $\varphi^{-1}$ on $\varphi(V_n)$, as well as the map which is equal to $\varphi^{-1}$ on $V_n$ and $\varphi$ on $\varphi^{-1}(V_n)$ both belong to $\Gamma_n$. Thus $y$ belongs to the same  $\Gamma_x(\varphi)$-orbit as $\varphi^{\pm 1}(y)$, so $\Gamma_x(\varphi)y= \{\varphi^n(y) \colon n \in \Z\}$. This completes the description of the $\Gamma_x(\varphi)$-orbits.

We still have to prove the second assertion. One inclusion comes from what we just established. Let $g \in F(\varphi)$ and set 
\[F= \lset{k \in \Z \colon \exists x \ g(x) = \varphi^k(x) } \] 
$F$ is a finite set; for any sufficiently large $n$, on any atom $U_{i,j}$ of $\mcA_n$ there exists $k \in F$ such that $g(y)= \varphi^k(y)$ for all $y \in U_{i,j}$.  If $g$ does not belong to $\Gamma_x(\varphi)$, there must exist some $k \in F$ such that for infinitely many $n$ there exists some atom $U_{i,j}$ of $\mcA_n$ such that either $k+j \ge n_i$ or $k+j < 0$. Considering only a subsequence of $(\mcA_n)$, we may assume that we are always in the first case and $n_i-j$ is constant equal to 
some integer $m$ (since $0 \le n_i-j \le k$, only finitely many values are possible); or always in the second case and $j$ is constant. In the first case, $g$ maps $\varphi^{-m-1}(x) \in O^{-}(x)$ to $\varphi^{k-m}(x) \in O^+(x)$; in the second case $g$ maps $\varphi^j(x) \in O^+(x)$ to $\varphi^{k+j}(x) \in O^-(x)$. 

We just proved that if $g \in  F(\varphi) \setminus \Gamma_x(\varphi)$ then $g(O^+(x)) \neq O^+(x)$, which is the contrapositive of the implication we were aiming for.
\end{proof}

\section{Ample groups and a pointed version of Krieger's theorem}\label{s:ample}
Now we go over some notions from Krieger \cite{Krieger1979}. We will in particular establish a strengthening of the main result of \cite{Krieger1979} (theorem \ref{t:Krieger_pointed} below). 
\subsection{Ample groups}

\begin{defn}[Krieger \cite{Krieger1979}]
A subgroup $\Gamma$ of $\Homeo(X)$ is an \emph{ample group} if 
\begin{itemize}
    \item $\Gamma$ is a locally finite, countable, full group.  
    \item For all $\gamma \in \Gamma$, $\{x \colon \gamma(x)=x\}$ is clopen.
\end{itemize}
\end{defn}

Our main example comes from the objects introduced in the previous section: given a minimal homeomorphism $\varphi$ and $x \in X$, the group $\Gamma_x(\varphi)$  is an ample group (the fact that for each $\gamma$ the set $\{x \colon \gamma(x) =x\}$ is clopen comes from the fact that $\varphi$ has no periodic points, since it acts minimally). Actually, all ample groups are of this form, see Theorem \ref{t:ample_vs_integers}.

\begin{defn}[Krieger \cite{Krieger1979}]
Let $\Gamma$  be a subgroup of $\Homeo(X)$, and $\mcA$ be a Boolean subalgebra of $\Clopen(X)$. We say that $(\mcA,\Gamma)$ is a \emph{unit system} if:
\begin{itemize}
    \item Every $\gamma \in \Gamma$ induces an automorphism of $\mcA$.
    \item If $\gamma \in \Gamma$ is such that $\gamma(A)=A$ for all $A \in \mcA$, then $\gamma=1$.
    \item If $g \in \Homeo(X)$ induces an automorphism of $\mcA$, and for any atom of $\mcA$ there exists $\gamma_A \in \Gamma$ such that ${\gamma_{A}}_{|A}=g_{|A}$, then $g \in \Gamma$.
\end{itemize}
We say that $(\mcA,\Gamma)$ is a finite unit system if $\mcA$ is finite (in which case $\Gamma$ is finite also).
\end{defn}

By definition, if we know how $\gamma \in \Gamma$ acts on atoms of $\mcA$ then $\gamma$ is uniquely determined.
We say that a unit system $(\mcB,\Delta)$ \emph{refines} another unit system $(\Gamma,\mcA)$ if $\Gamma \subseteq \Delta$ and $\mcA \subseteq \mcB$.

\begin{lemma}[Krieger]\label{l:Krieger_unit_systems}
Let $\Gamma$ be an ample group. There exists a refining sequence $(\mcA_n,\Gamma_n)$ of finite unit systems such that 
\[\Clopen(X) = \bigcup_n \mcA_n \quad \text{ ;} \quad \Gamma= \bigcup_n \Gamma_n\]
We say that such a sequence of unit systems is \emph{exhaustive}.
\end{lemma}

\begin{proof}
Fix enumerations $(\gamma_n)_{n\in\N}$ of $\Gamma$ and $(U_n)_{n\in\N}$ of $\Clopen(X)$. We construct inductively a refining sequence of
finite unit systems $(\mcA_n,\Gamma_n)_{n\in\N}$ such that 
$$\forall n\in\N, U_n\in \mcA_n \text{ and } \gamma_n\in \Gamma_n.$$
Assume $\gamma_0=id$ and $U_0=X$, so that $(\{X,\emptyset\},\{\gamma_0\})$ is already a unit system, and set $\mcA_0=\{X,\emptyset\}$, 
$\Gamma_0=\{\gamma_0\}$.

Then assume $(\mcA_n,\Gamma_n)$ is constructed for some $n\geq 0$, and let $\Gamma'_n$ denote the (finite) group generated by $\Gamma_n$ and $\gamma_{n+1}$.
Fix $\gamma \in \Gamma_n'$. For any integer $p$, the set $U_p(\gamma)$ of points which have period exactly $p$ for $\gamma$ is clopen. Let $I(\gamma)$ denote the (finite) set of all $p$ such that $U_p(\gamma) \ne \emptyset$.

Whenever $p \ge 2$ and $x \in U_p(\gamma)$, we can find some clopen neighborhood $V_p$ of $x$ such that $V_p, \gamma(V_p),\ldots,\gamma^{p-1}(V_p)$ are pairwise disjoint. Since $U_p(\gamma)$ is covered by finitely many such $V_p$, we can then produce a clopen $W_p(\gamma)$ such that for any $p \in I(\gamma)$ one has
$$U_p(\gamma) = \bigsqcup_{k=0}^{p-1} \gamma^k(W_p(\gamma))  $$

Then the family $(\gamma^k(W_p(\gamma))_{p \in I, 0 \le k \le p-1}$ forms a clopen partition, which generates a finite Boolean subalgebra $\mcB_\gamma$ of $\Clopen(X)$. For any atom $A$ of this partition, either $\gamma$ is equal to the identity on $A$, or $\gamma(A) \cap A= \emptyset$.

Let $\mcA_{n+1}$ denote the coarsest $\Gamma_n'$-invariant Boolean subalgebra of $\Clopen(X)$ which refines each $\mcB_\gamma$ as well as $\mcA_n$ and the subalgebra $\{\emptyset, U_{n+1}, X \setminus U_{n+1}, X\}$. It is a finite subalgebra of $\Clopen(X)$.

Finally, let $\Gamma_{n+1}$ denote all $\gamma \in \Gamma$ such that for any atom $U$ of $\mcA_{n+1}$ there exists $\delta \in \Gamma_n'$ which coincides with $\gamma$ on $U$.

By construction, $\Gamma_n'$ is contained in $\Gamma_{n+1}$, and $U_{n+1} \in \mcA_{n+1}$. It remains to point out that $(\mcA_{n+1},\Gamma_{n+1})$ is a unit system; to see this, choose $\gamma \in \Gamma_n'$ and $U$ an atom of $\mcA_{n+1}$ such that $\gamma(U)=U$. Then $U$ is contained in some atom $A$ of $\mcB_\gamma$, and either $\gamma$ coincides with the identity on $A$ or $\gamma(A) \cap A= \emptyset$. Since $U \subseteq A$ and $\gamma(U)=U$ we must be in the first situation, which concludes the proof.
\end{proof}

The following result is an analogue of a lemma due to Glasner and Weiss \cite{Glasner1995a}*{Lemma~2.5 and Proposition~2.6}; the argument we use in the proof is essentially the same as in \cite{Glasner1995a}. 
\begin{lemma}\label{l:ample_Glasner_Weiss}
Let $\Gamma$ be an ample group; recall that $M(\Gamma)$ is the set of all $\Gamma$-invariant Borel probability measures on $X$. Let $A,B$ be two clopen subsets of $X$.
\begin{enumerate}
    \item \label{proof:GW1} Assume that $\mu(A) < \mu(B)$ for all $\mu \in M(\Gamma)$. Then there exists $\gamma \in \Gamma$ such that $\gamma(A) \subset B$.
    \item \label{proof:GW2} Assume that $\mu(A)=\mu(B)$ for all $\mu \in M(\Gamma)$ and that $\Gamma$ acts topologically transitively. Then there exists $g \in [R_\Gamma]$ such that $g(A)=B$.
\end{enumerate}

\end{lemma}

We only assume that $\Gamma$ acts topologically transitively above because that is the natural hypothesis to make the argument work
. Since our concern is with minimal actions, we mostly stick with the minimality assumption throughout the paper but make an exception here.

\begin{proof}
\eqref{proof:GW1} Find an exhaustive sequence $(\mcA_n,\Gamma_n)$ of finite unit systems. There exists $m \in \N$ such that for all $n\ge m$ both $A$ and $B$ are unions of atoms of $\mcA_n$. For $U$ an atom of $\mcA_n$, we may thus consider 
$$a_n(U)= \left|\lset{ V \in \Gamma_n U \colon V \subset A }\right| $$
$$b_n(U)= \left|\lset{ V \in \Gamma_n U \colon V \subset B }\right| $$
(where $|F|$ stands for the cardinality of a finite set $F$)

Assume that for any $n \ge m$ there exists $ p\ge n$ and an atom $U_p \in \mcA_p$ such that $a_p(U_p) \ge b_p(U_p)$; pick $x_p \in U_p$ and set 
$$\mu_p= \frac{1}{|\Gamma_p|} \sum_{\gamma \in \Gamma_p} \delta_{\gamma(x_p)} $$
where $\delta_y$ stands for the Dirac measure at $y \in X$. Then $\mu_p$ is a $\Gamma_p$-invariant Borel probability measure for all $p$, and $\mu_p(A) \ge \mu_p(B)$. Using compactness of the space of all Borel probability measures on $X$, we may take a cluster point $\mu$ of $(\mu_p)$, and $\mu$ is a $\Gamma$-invariant measure such that $\mu(A) \ge \mu(B)$, contradicting our assumption. 

Hence we see that, for any sufficiently large $n \ge m$, any atom $U$ of $\mcA_n$ is such that $a_n(U) < b_n(U)$. From this we obtain the existence of $\gamma \in \Gamma_n$ such that $\gamma(A) \subset B$ (any permutation of the atoms of a column of $\mcA_n$ is induced by an element of $\Gamma_n$).

\eqref{proof:GW2} We may (and do) assume that $A,B$ are nonempty and $A \cap B = \emptyset$; we fix $a \in A$ and $b \in B \cap \Gamma a$ (here we are using topological transitivity). 

Fix a compatible ultrametric $d$ on $X$; we use a back-and-forth argument to build sequences of clopen subsets $(U_n)$, $(V_n)$ of $X$, and a sequence $(\gamma_n)$ of elements of $\Gamma$ such that for all $n$:
\begin{itemize}
    \item  $U_n \subseteq A$, $a \in A \setminus U_n$, and $U_n \cap \bigcup_{i=0}^{n-1} U_i = \emptyset$.
    \item $V_n \subseteq B$, $b \in A \setminus V_n$, and $V_n \cap \bigcup_{i=0}^{n-1} V_i = \emptyset$.
    \item The diameters of $A \setminus \bigcup_{i=0}^n U_i$ and $B \setminus \bigcup_{i=0}^n V_i$ converge to $0$.
    \item For all $n$, $\gamma_n(U_n)=V_n$.
\end{itemize}
Assuming that this is indeed possible, we obtain the desired $g$ by setting $g=\gamma_n$ on $U_n$, $g=\gamma_n^{-1}$ on $V_n$, $g(a)=b$, $g(b)=a$ (and $g$ is the identity outside $A \cup B$).

We use even steps of the process to make the diameter of $X \setminus \bigcup_{i=0}^n U_i$ decrease, and odd steps to control $X \setminus \bigcup_{i=0}^n V_i$; since the conditions are symmetric, let us explain what we do when $U_0,V_0,\ldots,U_{n-1}, V_{n-1}$ have been defined and $n$ is even. Set 
\[\tilde A= A \setminus \bigsqcup_{i=0}^{n-1} U_i \ , \quad \tilde B=B \setminus \bigsqcup_{i=0}^{n-1} V_i\]
Then $a \in \tilde A$, $b \in \tilde B$, and $\mu(\tilde A)=\mu(\tilde B)$ for all $\mu \in M(\Gamma)$.

Pick $\varepsilon >0$ such that $B(a,\varepsilon) \subset \tilde A$. If follows from Lemma \ref{l:props_invariant_measures} that for small enough $\delta >0$, we have 
\[\sup_{\mu \in M(\Gamma)} \mu\left(\tilde A \setminus B(a,\varepsilon)\right) < \inf_{\mu \in M(\Gamma)} \mu \left(\tilde B \setminus B(b,\delta)\right)\]
Set $U_{n+1}= \tilde A \setminus B(a,\varepsilon)$; by \eqref{proof:GW1} we can find $\gamma_{n+1} \in \Gamma$ such that $\gamma_{n+1}(U_{n+1}) \subset \tilde B \setminus B(b,\delta))$. We set $V_{n+1}= \gamma(U_{n+1})$ and move on to the next step.
\end{proof}





For future reference, we state the original lemma of Glasner and Weiss, whose proof is very similar to the proof of Lemma \ref{l:ample_Glasner_Weiss}.

\begin{lemma}[Glasner--Weiss \cite{Glasner1995a}*{Lemma~2.5}]\label{l:original_Glasner_Weiss}
Let $\varphi$ be a minimal homeomorphism, and let $A,B$ be two clopen sets such that $\mu(A)< \mu(B)$ for any $\mu \in M(\varphi)$. 

Then there exists an integer $N$ such that, whenever $\mcA$ is a Kakutani--Rokhlin partition compatible with $A,B$ and such that all columns of $\mcA$ have height $\ge N$, for any column $\mcC$ of $\mcA$ one has
$$\left|\lset{\alpha \in \mcC \colon \alpha \subseteq A }\right| < \left|\lset{\alpha \in \mcC \colon \alpha \subseteq B} \right| $$
\end{lemma}

We reformulate Lemma \ref{l:ample_Glasner_Weiss} using the Polish group topology of $\Homeo(X)$; we recall that this topology can be viewed either as the topology of compact-open convergence on $X$, or as the permutation group topology induced by the action of $\Homeo(X)$ on $\Clopen(X)$. Using this last point of view, a neighborhood basis of the identity is given by the subgroups 
$$G_\mcA = \lset{ g \in \Homeo(X) \colon \forall A \in \mcA \ g(A)=A } $$
where $\mcA$ ranges over all clopen partitions of $X$.

\begin{lemma}\label{l:closure_full_group}
Assume that $\Gamma$ is a minimal ample group. Then 
\[\overline{[R_\Gamma]} = \lset{g \in \Homeo(X) \colon \forall \mu \in M(\Gamma) \ g_* \mu=\mu}\]
\end{lemma}

\begin{proof}
Any element $g$ of $[R_\Gamma]$ must be such that $g_* \mu = \mu$ for all $\mu \in M(\Gamma$) (see Lemma \ref{l:invariant_measures_full_group}) and $\{g \colon g_* \mu=\mu\}$ is a closed subset of $\Homeo(X)$ for all $\mu$. This proves the inclusion from left to right.

To see the converse inclusion, take $g$ such that $g_* \mu= \mu$ for all $\mu \in M(\Gamma)$. Let $U$ be a neighborhood of $g$ in $\Homeo(X)$; by definition of the topology of $\Homeo(X)$, there exists a clopen partition $\mcA$ of $X$ such that 
\[\lset{h \in \Homeo(X) \colon \forall A \in \mcA \ h(A)=g(A) } \subseteq U \]
Lemma \ref{l:ample_Glasner_Weiss} shows that for any $A \in \mcA$ there exists $h_A \in [R_\Gamma]$ such that $h_A(A)=g(A)$, whence there exists $h \in [R_\Gamma] \cap U$ obtained by setting $h(x)=h_A(x)$ for all $x \in A$ and all $A \in \mcA$. 
\end{proof}

The heart of the above argument is the fact that, for a full group $G$ contained in $\Homeo(X)$, the closure $\overline{G}$ consists of all homeomorphisms $h$ such that for any $A \in \Clopen(X)$ there exists $g \in G$ such that $h(A)=g(A)$.

The following lemma will help us deduce the classification theorem for minimal homeomorphisms from the classification theorem for minimal ample groups.

\begin{lemma}\label{l:closure_Gamma_x}
Let $\varphi$ be a minimal homeomorphism, and $x \in X$. 

Then $M(\Gamma_x(\varphi))=M(\varphi)$, and
 $$\overline{[R_{\Gamma_x(\varphi)}]}=\overline{[R_\varphi]}=\lset{g \in \Homeo(X) \colon \forall \mu \in M(\varphi) \ g_* \mu=\mu} $$
\end{lemma}

\begin{proof}
To simplify notation, denote $\Gamma= \Gamma_x(\varphi)$. 

Since $\Gamma \subset [R_\varphi]$, we have $[R_\Gamma] \subseteq [R_\varphi]$ and
$M(\varphi)=M([R_\varphi]) \subseteq M(\Gamma) $.

Pick $\mu \in M(\Gamma)$, and let $U \in \Clopen(X)$; assume first that $\varphi(U) \cap U = \emptyset$, and $U$ does not contain $\varphi^{-1}(x)$. Then the involution $\gamma$ equal to $\varphi$ on $U$, $\varphi^{-1}$ on $\varphi(U)$ and the identity elsewhere belongs to $\Gamma$. Thus $$\mu(\varphi(U))=\mu(\gamma(U))=\mu(U)$$
If $U \in \Clopen(X)$ is any clopen not containing $\varphi^{-1}(x)$, we can write it (by exhaustion) as a disjoint union of clopens $U_i$ such that $\varphi(U_i) \cap U_i = \emptyset$, so that 
$$\mu(\varphi(U))= \sum_{i=1}^n \mu(\varphi(U_i))=\sum_{i=1}^n \mu(U_i)=\mu(U)$$

Finally, if $\varphi^{-1}(x) \in U$, we can find a clopen $V$ containing $x$ and such that $V$ and $\varphi(V)$ both have arbitrarily small diameter (for some compatible distance), thus if we fix $\varepsilon >0$ we can find a clopen $V \subset U$ such that $\varphi^{-1}(x) \in V$ and $\mu(V)$, $\mu(\varphi(V))$ are both $< \varepsilon$. Since $\mu(U \setminus V)= \mu(\varphi(U) \setminus \varphi(V))$, we have 
$$|\mu(U) - \mu(\varphi(U))| = |\mu(V) - \mu(\varphi(V))| \le 2\varepsilon $$
This is true for any $\varepsilon >0$, so $\mu(\varphi(U))=\mu(U)$ for any clopen $U$: in other words, $\mu \in M(\varphi)$.

This establishes the first assertion; the second one is an immediate consequence of this, since then we have by Lemma \ref{l:closure_full_group} the equality
$$\overline{[R_{\Gamma_x(\varphi)}]} = \lset{g \in \Homeo(X) \colon \forall \mu \in M(\varphi) \ g_* \mu=\mu}$$
and the right-hand side of this equality is closed and contains $[R_\varphi]$.
\end{proof}

\subsection{A strengthening of Krieger's theorem}
We turn to the proof of a version of Krieger's theorem that is instrumental to our approach. The proof is based on a back-and-forth argument already present in Krieger's proof.

\begin{defn}
Let $\Gamma$ be a subgroup of $\Homeo(X)$. Given $U,V \in \Clopen(X)$ we write $U \sim_\Gamma V$ if there exists $\gamma \in \Gamma$ such that $\gamma(U)=V$.
\end{defn}

\begin{defn}
Let $\Gamma,\Lambda$ be two ample subgroups of $\Homeo(X)$. We say that $\Gamma$, $\Lambda$ have \emph{isomorphic closures} if there exists $g \in \Homeo(X)$ such that $g \overline{\Gamma} g^{-1}= \overline{\Lambda}$ or, equivalently,
\[\forall U,V \in \Clopen(X) \quad \left( U \sim_\Gamma V\right) \Leftrightarrow \left(g(U) \sim_\Lambda g(V) \right) \]

\end{defn}

The fact that both conditions in the previous definition are equivalent follows from the remark immediately following Lemma \ref{l:closure_full_group}; they are also equivalent (in our context) to what Krieger calls \enquote{isomorphism of dimension ranges}, though we formulate it in the way which we find most suitable for our purposes in this article.

\begin{defn}
Let $\Gamma$ be a minimal ample group. We say that a closed subset $K$ of $X$ is \emph{$\Gamma$-sparse} if each $\Gamma$-orbit intersects $K$ in at most one point. For such a $K$, we say that a finite unit system $(\mcA,\Sigma)$ with $\Sigma\le\Gamma$ is \emph{$K$-compatible} if any $\Sigma$-orbit (for the action of $\Sigma$ on the atoms of $\mcA$) has at most one element which intersects $K$. 
\end{defn}

Our aim in this subsection is to prove the following strengthening of Krieger's theorem \cite{Krieger1979}*{Theorem~3.5}. This result plays a crucial role in our proof of the classification theorem for minimal ample groups.

\begin{theorem}\label{t:Krieger_pointed}
Let $\Gamma$, $\Lambda$ be minimal ample groups with isomorphic closures. Let $K,L$ be closed subsets of $X$ such that $K$ is $\Gamma$-sparse, and $L$ is $\Lambda$-sparse. Assume that $h \colon K \to L$ is a homeomorphism.

Then there exists $g \in \Homeo(X)$ such that $g\Lambda g^{-1}= \Gamma$, and $g_{|K}=h$.
\end{theorem}

\begin{remark}
Although we do not need this result here, we note that whenever $\varphi$ is a minimal homeomorphism of $X$, $\Gamma_x(\varphi)$ and $\Gamma_y(\varphi)$ induce the same relation on $\Clopen(X)$ for any $x,y \in X$ (see \cite{Robert2019} for an elementary proof), and it then follows from Theorem \ref{t:Krieger_pointed} (with $K=L=\emptyset$) that they are conjugate.
\end{remark}

We begin working towards the proof of Theorem \ref{t:Krieger_pointed}. 

\begin{lemma}\label{l:suitable_unit_systems}
Let $\Gamma$ be a minimal ample group, and $K$ a $\Gamma$-sparse closed subset of $X$.

\begin{itemize}
    \item For any finite unit system $(\mcA,\Sigma)$ with $\Sigma\le\Gamma$, there exists a $K$-compatible finite unit system $(\mcB,\Sigma)$ with $\mcB$ refining $\mcA$.
    \item There exists a refining sequence $(\mcB_n,\Gamma_n)$ of $K$-compatible finite unit systems such that $\Clopen(X)=\bigcup \mcB_n$ and $\Gamma=\bigcup \Gamma_n$.

\end{itemize}

\end{lemma}

\begin{proof} 
Fix a finite unit system $(\mcA,\Sigma)$ with $\Sigma \le \Gamma$.

Choose $U_1,\ldots,U_k$ representatives of the $\mcA$-orbits. For each $i$, the $\mcA$-orbit of $U_i$ is of the form $U_i \sqcup \gamma_{i,1} U_i \sqcup \gamma_{i,k_i} U_i$. For all $x$ in $U_i$, at most one element of $\{x,\gamma_{1,1}(x),\ldots,\gamma_{i,k_i}(x)\}$ can belong to $K$. 

So we can write $U_i = \bigsqcup_{j=1}^{n_i} U_{i,j}$, where, for every $j$, $U_{i,j} \sqcup \gamma_{i,1} U_{i,j} \ldots \sqcup \gamma_{i,k_i} U_{i,j}$ intersects $K$ in at most one point. 

Finally, let $\mcB$ be the algebra whose atoms are the $\gamma (U_{i,j})$, $\gamma \in \Sigma$.

This proves the first part of the lemma's statement; the second part immediately follows from the first part and Lemma \ref{l:Krieger_unit_systems}.
\end{proof}

\begin{lemma}\label{l:clopen_intersecting_K}
Let $\Gamma$ be a minimal ample group, and $K$ be a $\Gamma$-sparse closed subset of $X$.

Let $U$ be a nontrivial clopen subset of $X$, and $A$ be a clopen subset of $K$. Let $V \in \Clopen(X)$ be such that $A \subset V \cap K$ and $\mu(U) < \mu(V)$ for all $\mu \in M(\Gamma)$.

There exists $U' \in \Clopen(X)$ such that $U' \cap K=A$, $U' \subset V$, and $U' \sim_\Gamma U$.
\end{lemma}

\begin{proof}
Fix some integer $N$. Let $(\Gamma_n)$ be an increasing sequence of finite groups such that $\bigcup_n \Gamma_n=\Gamma$. Assume first, for a contradiction, that for all $n$ there exists $x_n \in X$ such that $\Gamma_n x_n$ has cardinality $<N$. By compactness, we may assume that $(x_n)$ converges to $x$; since $\Gamma$ acts aperiodically, we can find $\gamma_1,\ldots,\gamma_N$ such that $\gamma_i (x) \ne \gamma_j(x)$ for all $i \ne j$. Then we also have $\gamma_i(x_n) \ne \gamma_j(x_n)$ for $n$ large enough and $i \ne j$, which is the desired contradiction.

Thus, for any $n$ large enough, every $\Gamma_n$-orbit has cardinality $\ge N$. Thus there exists a finite unit system $(\mcA,\Sigma)$ with $\Sigma \le \Gamma$ such that every orbit for the action of $\Sigma$ on the set of atoms of $\mcA$ has cardinality $\ge N$.
This allows us to find $\gamma \in \Gamma$ such that every $x \in X$ has a $\gamma$-orbit of cardinality $\ge N$. 

Since $K, \ldots, \gamma^{N-1}(K)$ are closed and pairwise disjoint, we can find some disjoint clopens $B_i$ such that $\gamma^i(K) \subseteq B_i$ for all $i \in \{0,\ldots,N-1\}$. Then $B=B_0 \cap \bigcup_{i=1}^{N-1} \gamma^{-i}(B_i)$ is clopen, contains $K$, and $B, \ldots, \gamma^{N-1}(B)$ are pairwise disjoint.

So there exists $B \in \Clopen(X)$ containing $K$ and such that $\mu(B) \le \frac{1}{N}$ for all $\mu \in K$. Hence we can find $B$ clopen, containing $K$, and such that $\mu(B) < \mu(U)$ for all $\mu \in M(\Gamma)$.

There exists $C \in \Clopen(X)$ such that $A= K\cap C$, so $D=B \cap C$ is clopen, $D \cap K=A$, and $\mu(D)< \mu(U)$ for all $\mu \in M(\Gamma)$. Similarly we can find $E \in \Clopen(X)$ such that $E \cap K = K \setminus A$, $\mu(E)< \mu(X \setminus U)$ for all $\mu \in M(\Gamma)$, and $E \cap D= \emptyset$. 

There exists $\gamma \in \Gamma$ such that $\gamma(D) \subset U$, $\gamma(E) \subset X \setminus U$; set $W= \gamma^{-1}(U)$. We have $W \cap K=A$ and $W \sim_\Gamma U$.

Set $V'= V \cap W$; it is clopen, contained in $V$, and $V' \cap K=A$. Since there are clopen subsets of arbitrarily small measures containing $K$, there exists a clopen $V''$ contained in $V$, disjoint from $K$, and such that $\mu(U)< \mu(V'')$ for all $\mu \in M(\Gamma)$. So $\mu(W \setminus V') < \mu(V'' \setminus V')$ for all $\mu \in M(\Gamma)$. Hence there exists $Y \sim_\Gamma W \setminus V'$ such that $Y \subset V'' \setminus V'$. We may finally set $U'= V' \sqcup Y$.
\end{proof}

\begin{nota*}
For the remainder of this section, we fix two minimal ample groups $\Gamma$, $\Lambda$, and assume that $\sim_\Gamma$ and $\sim_\Lambda$ coincide (we reduce to this situation by conjugating $\Lambda$ if necessary). We denote this equivalence relation on $\Clopen(X)$ by $\sim$. We also fix closed subsets $K,L$ such that $K$ is $\Gamma$-sparse and $L$ is $\Lambda$-sparse, and a homeomorphism $h \colon K \to L$.
\end{nota*}

\begin{defn}
Let $\Delta$ be a finite subgroup of $\Gamma$, $\Sigma$ a finite subgroup of $\Lambda$, and assume that $(\mcA,\Delta)$, $(\mcB,\Sigma)$ are finite unit systems.

A Boolean algebra isomorphism $\Phi \colon \mcA \to \mcB$ \emph{respects $\sim$} if for any $A \in \mcA$ one has $\Phi(A) \sim A$. We say that $\Phi$ \emph{conjugates} $(\mcA,\Delta)$ on $(\mcB,\Sigma)$ if $\Sigma_{|\mcB}= \Phi \Delta_{|\mcA} \Phi^{-1}$.
\end{defn}

\begin{defn}
Let $\Delta$ be a finite subgroup of $\Gamma$, $\Sigma$ a finite subgroup of $\Lambda$, and assume that $(\mcA,\Delta)$ is a $K$-compatible finite unit system, and $(\mcB,\Sigma)$ is a $L$-compatible finite unit system.

We say that a Boolean algebra isomorphism $\Phi \colon \mcA \to \mcB$ is $h$-compatible if:
\begin{enumerate}
    \item $\Phi$ respects $\sim$.
    \item $\Phi$ conjugates $(\mcA,\Delta)$ on $(\mcB,\Sigma)$.
    \item For every atom $\alpha \in \mcA$, $\Phi(\alpha) \cap L= h(\alpha \cap K)$.
\end{enumerate}

\end{defn}

\begin{lemma}\label{l:Krieger_construction}
Assume that $(\mcA,\Delta)$, $(\mcB,\Sigma)$ are respectively $K$- and $L$-compatible finite unit systems with $\Delta \le \Gamma$, $\Sigma \le \Lambda$, and $\Phi \colon \mcA \to \mcB$ is a $h$-compatible Boolean algebra isomorphism. 

Let $(\mcA',\Delta')$ be a $K$-compatible finite unit system refining $(\mcA,\Delta)$ with $\Delta' \le \Gamma$. 

Then one can find a $L$-compatible finite unit system $(\mcB',\Sigma')$ refining $(\mcB,\Sigma)$, with $\Sigma' \le \Lambda$ and a $h$-compatible isomorphism $\Phi' \colon \mcA' \to \mcB'$ which extends $\Phi$.
\end{lemma}

\begin{proof}
For every orbit $\rho$ of the action of $\Delta$ on the atoms of $\mcA$, we choose a representative $A_\rho$. If $\rho$ intersects $K$, we choose $A_\rho$ so that $A_\rho \cap K \ne \emptyset$ (and it is the unique such atom in $\rho$, because $\mcA$ is $K$-compatible).

For every $A \in \rho$, we denote by $\delta(\rho,A)$ the element of $\Delta$ which maps $A$ to $A_\rho$, $A_\rho$ to $A$, and is the identity everywhere else. This is an involution (and it is uniquely defined by definition of a unit system); in the particular case where $A=A_\rho$ we have $\delta(\rho,A_\rho)=id$. Similarly, we denote $\sigma(\rho,A)$ the involution of $\Sigma$ exchanging $\Phi(A)$ and $\Phi(A_\rho)$ and which is the identity everywhere else.

For every $\rho$ we have 
\[A_\rho= \bigsqcup_{C \in \text{ atoms}(\mcA') \colon C \subseteq A_\rho} C\]

Let $C_1,\ldots,C_p$ denote the atoms of $\mcA'$ contained in $A_\rho$. Applying Lemma \ref{l:clopen_intersecting_K}, find a clopen $U(C_1) \sim_\Gamma C_1$ contained in $\Phi(A_\rho)$ and such that $U(C_1) \cap L = h(C_1 \cap K)$; then a clopen $U(C_2) \sim_\Gamma C_2$ contained in $\Phi(A_\rho)$ disjoint from $C_1$ and such that $U(C_2) \cap L = h(C_2 \cap K)$; and so on.

We now have
\[\Phi(A_\rho) = \bigsqcup_{C \in \text{ atoms}(\mcA') \colon C \subseteq A_\rho} U(C)\]
where $U(C) \sim C$, and $U(C) \cap L = h(C \cap K)$ for all $C$. 

\begin{figure}[ht]
    \centering
    \includegraphics[height=5cm]{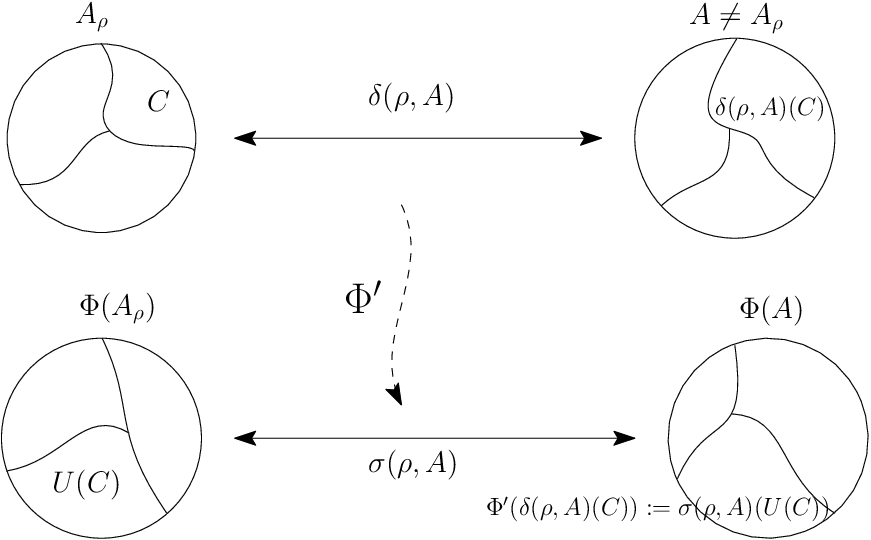}
    \caption{Construction of $\mcB'$ and $\Phi'$}
    \label{fig: krieger construction of Phi'}
\end{figure}

We define the algebra $\mcB'$ by setting as its atoms all $U(C)$, for $C$ an atom of $\mcA'$ contained in some $A_\rho$, as well as all $\sigma(\rho,A)(U(C))$ for $A \in \rho$ and $C$ contained in $A_\rho$ (see figure \ref{fig: krieger construction of Phi'}). We obtain an isomorphism $\Phi' \colon \mcA' \to \mcB'$ by setting $\Phi(C)=U(C)$ for every atom of $\mcA'$ contained in some $A_\rho$; and then
for any atom $C$ of $\mcA'$ contained in some $A \in \mcA$ whose $\Delta$-orbit is $\rho$, 
\[\Phi'(C)= \sigma(\rho,A) (U(\delta(\rho,A)(C))) \]

\begin{figure}[ht]
    \centering
    \includegraphics[width=.9\linewidth]{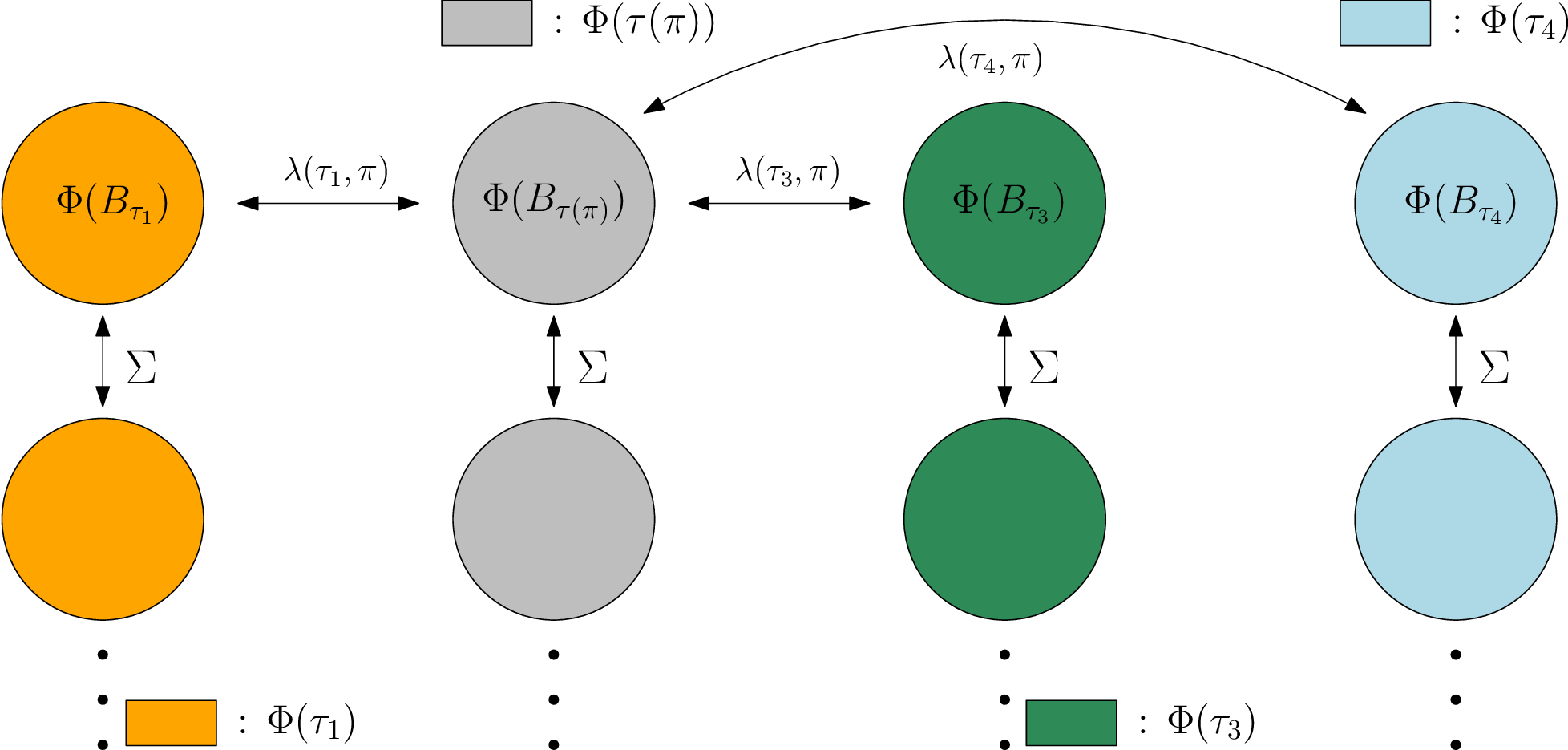}
    \caption{Construction of $\Sigma'$ imitating the behavior of $\Delta'$ on the image of a $\Delta'$-orbit $\pi$ containing four $\Delta$-orbits}
    \label{fig: Krieger construction Sigma'}
\end{figure}

We now need to construct the group $\Sigma'$. In the remainder of the proof, the letter $\tau$ always stands for an orbit of the action of $\Delta$ on the atoms of $\mcA'$, and the letter $\pi$ for an orbit of the action of $\Delta'$ on the atoms of $\mcA'$. For any $\tau$ there exists a unique $\pi$ which contains $\tau$. 

For any $\tau$ we choose a representative $B_\tau$, and among all $B_\tau$ contained in a given $\pi$ we choose one $B_{\tau(\pi)}$. For every $\tau$ contained in $\pi$, we choose an involution $\lambda(\tau,\pi)\in \Lambda$ mapping $\Phi'(B_{\tau(\pi)})$ to $\Phi'(B_\tau)$, and equal to the identity elsewhere. Let $\Sigma'$ be the group generated by $\Sigma$ and $\{\lambda(\tau,\pi) \colon \tau \subset \pi\}$. Then $(\mcB',\Sigma')$ is a finite unit system (because we have added at most one link between any two $\Sigma$-orbits) and $\Phi'$ conjugates $(\mcA',\Delta')$ to $(\mcB',\Sigma')$.

For every atom $A$ of $\mcA'$ we have $\Phi(A) \cap L= h(A \cap K)$ by choice of $U(A)$. Since $\Phi'$ conjugates $(\mcA',\Delta')$ to $(\mcB',\Sigma')$, for any two atoms $C,D$ of $\mcA'$ which intersect $K$, $\Phi'(C)$ and $\Phi'(D)$ belong to different $\Sigma'$-orbits. This proves that $(\mcB',\Sigma')$ is $L$-compatible, and completes the proof.
\end{proof}

\begin{proof}[End of the proof of Theorem \ref{t:Krieger_pointed}.]
Fix sequences $(\mcA_n,\Gamma_n)$ and $(\mcB_n,\Lambda_n)$ of respectively $K$ and $L$-compatible finite unit systems as in Lemma \ref{l:suitable_unit_systems}, with $\mcA_0=\mcB_0= \{\emptyset, X\}$ and $\Gamma_0=\Lambda_0=\{\mathrm{id}\}$.

Then $(\mcA_0,\Gamma_0)$, $(\mcB_0,\Lambda_0)$ are respectively $K$ and $L$-compatible finite unit systems, and $\Phi_0=id$ is $h$-compatible.

Applying Lemma \ref{l:Krieger_construction} we build inductively refining sequences $(\mcA_n',\Gamma_n')$, (resp. $(\mcB_n',\Lambda_n')$)  of $K$-compatible (resp.~$L$-compatible) finite unit systems contained in $(\Clopen(X),\Gamma)$ and $(\Clopen(X),\Lambda)$ respectively, as well as $h$-compatible isomorphisms $\Phi_n$ conjugating $(\mcA_n',\Gamma_n')$ to $(\mcB_n',\Lambda_n')$  such that for odd $n$ $(\mcA_n',\Gamma_n')$ refines $(\mcA_n,\Gamma_n)$, and for even $n$ $(\mcB_n',\Lambda_n')$ refines $(\mcB_n,\Lambda_n)$  



To see why this is possible, assume that we have carried out this construction up to some even $n$ (the odd case is symmetric). We pick $k \ge n+1$ such that $(\mcA_k,\Gamma_k)$ refines $(\mcA_n',\Gamma_n')$. The unit system $(\mcA_k,\Gamma_k)$ (resp. $(\mcB_n',\Lambda_n')$) is $K$-compatible (resp. $L$-compatible), so applying Lemma \ref{l:Krieger_construction} gives us some $L$-compatible finite unit system $(\mcB_{n+1}',\Lambda_{n+1}')$ contained in $(\Clopen(X),\Lambda)$ which refines $(\mcB_n',\Lambda_n)$, and a $h$-compatible isomorphism $\Phi_{n+1} \colon \mcA_k \to \mcB_{n+1}'$. Setting $\mcA_{n+1}'=\mcA_k$, $\Gamma_{n+1}'=\Gamma_k$, we are done.

This construction produces an isomorphism $\Phi$ of $\Clopen(X)$ (the union of the sequence $\Phi_n$) such that $\Phi \Gamma \Phi^{-1}= \Lambda$ (here, $\Gamma$ and $\Lambda$ are viewed as subgroups of the automorphism group of the Boolean algebra $\Clopen(X)$). By Stone duality, there exists a unique $g \in \Homeo(X)$ such that $g(U)=\Phi(U)$ for any $U \in \Clopen(X)$, and we have $g \Gamma g^{-1}= \Lambda$.

Clopen subsets of the form $\alpha \cap K$, where $\alpha$ is an atom of some $\mcA'_n$, generate the topology of $K$. For each such clopen we have $\Phi(\alpha) \cap L= h (\alpha \cap K)$, in other words $g(\alpha) \cap L = h(\alpha \cap K)$. It follows that $g_{|K}=h$, and we are done.
\end{proof}

\begin{defn}
We say that an ample group $\Gamma$ is \emph{saturated} if $\overline{\Gamma}= \overline{[R_\Gamma]}$.
\end{defn}

The classification theorem for minimal actions of saturated ample groups follows immediately from Krieger's theorem.

\begin{theorem}\label{t:Krieger_gives_saturated_GPS}
Let $\Gamma$, $\Lambda$ be two saturated ample subgroups of $\Homeo(X)$ acting minimally. Then the following conditions are equivalent:
\begin{enumerate}
 	\item \label{Krieger_saturated_0} $\Gamma$ and $\Lambda$ are conjugated in $\Homeo(X)$.
    \item \label{Krieger_saturated_1} $\Gamma$ and $\Lambda$ are orbit equivalent.
    \item \label{Krieger_saturated_2} There exists $g \in \Homeo(X)$ such that $g_*M(\Gamma)=M(\Lambda)$.
\end{enumerate}
\end{theorem}

\begin{proof}
Clearly the first condition implies the second, and we already know that the second implies the third. Now, assume that $\Gamma$, $\Lambda$ are saturated ample groups acting minimally, and $M(\Gamma)=M(\Lambda)$ (as usual, we reduce to this situation by conjugating $\Lambda$ if necessary).

Then we have $\overline{[R_\Gamma]}= \overline{[R_\Lambda]}$ (see Lemma \ref{l:closure_full_group}), hence also $\overline{\Gamma}=\overline{\Lambda}$. So $\Gamma$, $\Lambda$ are conjugated by Theorem \ref{t:Krieger_pointed}.
\end{proof}

\section{Balanced partitions}\label{s:balanced}

In this section, we fix an ample group $\Gamma$ acting minimally, and an exhaustive sequence $(\mcA_n,\Gamma_n)$ of finite unit systems. 

\begin{defn} We consider two equivalence relations on $\Clopen(X)$, defined as follows:
\begin{itemize}
\item $U \sim_\Gamma  V$ iff there exists $g \in \Gamma$ such that $g(U)=V$.
\item $U \sim_\Gamma^* V$ iff $\mu(U)=\mu(V)$ for any $\mu \in M(\Gamma)$ (equivalently, there exists $g \in [R_\Gamma]$ such that $g(U)=V$).
\end{itemize}

\end{defn}

The main difficulty in our proof of the classification theorem comes from the fact that $\sim_\Gamma^*$  may be strictly coarser than $\sim_\Gamma$.

\begin{defn}
An equivalence relation $\simeq$ on $\Clopen(X)$ is \emph{full} if for any clopens $A, B$, whenever $A = \bigsqcup A_i$, $B=\bigsqcup B_i$ and $A_i \simeq B_i$ for all $i$, we have $A \simeq B$.
\end{defn}

The relation $\sim_\Gamma$ is full: indeed, consider $(A_i)$, $(B_i)$ as above; applying Lemma \ref{l:Krieger_unit_systems} we find a finite unit system $(\mcA,\Delta)$ with $\Delta$ a finite subgroup of $\Gamma$, such that for all $i$ $A_i$ and $B_i$ are unions of atoms of $\mcA$, and there exists $\delta_i \in \Delta$ such that $\delta_i A_i = B_i$. We then see that there exists $\delta \in \Delta$ such that $\delta (\bigsqcup A_i)= \bigsqcup B_i$. 

It is immediate that $\sim_\Gamma^*$ is full. We note the following question: is $\sim_\Lambda$ full whenever $\Lambda$ is the full group associated to an action of a countable group ?

\begin{defn}
Let $\simeq$ be a full equivalence relation on $\Clopen(X)$. 

A \emph{$\simeq$-partition} of $X$ is a clopen partition $\mcA=(A_{i,j})_{(i,j) \in I}$ such that 
$$\forall i,j,k \quad \left( (i,j) \in I \text{ and } (i,k) \in I \right) \Rightarrow A_{i,j} \simeq A_{i,k}$$

We denote $I_i$ the set $\{j \colon (i,j) \in I\}$. For any element $\alpha=A_{i,j}$ of $\mcA$, the set $\{A_{i,k} \colon k \in I_i\}$ is called the \emph{$\mcA$-orbit} of $\alpha$ and denoted $O(\alpha)$.

Given an orbit $O=(A_{i,j})_{j \in I_i}$, a family of clopens $(B_{i,j})_{j\in I_i}$ such that 
$$\forall j\in I_i \quad B_{i,j}\subset A_{i,j} \text{ and } \forall j,k\in I_i \quad B_{i,j}\simeq B_{i,k}$$ 
is called a \emph{fragment} of $O$.
\end{defn}

Fragments of orbits are the abstract counterpart to the copies of columns that appear during the cutting procedure for Kakutani--Rokhlin partitions. For $A\subseteq \alpha$, with $\alpha$ an atom of $\mcA$, we may abuse notation and talk about $O(A)$ to designate a fragment of $O(\alpha)$ containing $A$ (such a fragment is obviously not unique, but we usually do not care about the particular choices we are making when selecting the atoms of our fragment). 

\begin{defn} Let $\simeq$ be a full equivalence relation on $\Clopen(X)$.

A $\simeq $-partition $\mcA$ is \emph{compatible} with a given clopen set $U$ if $U$ is a union of elements of $\mcA$ (equivalently, if $U$ belongs to the Boolean algebra generated by $\mcA$).

If $\mcA$ is compatible with $U$, and $O$ is a $\mcA$-orbit, we let $n_O(U)$ be the number of elements of $O$ which are contained in $U$.
\end{defn}

\begin{defn}
Let $\simeq$ be a full equivalence relation on $\Clopen(X)$, and $\mcA$, $\mcB$ be two $\simeq$-partitions.  

We say that $\mcB$ \emph{refines} $\mcA$ if:
\begin{itemize}
\item[•] $\mcB$ refines $\mcA$ as a clopen partition;   
\item[•] Whenever $\alpha$, $\beta$ belong to the same $\mcA$-orbit, we have $n_O(\alpha)=n_O(\beta)$ for any $\mcB$-orbit $O$. By analogy with Kakutani--Rokhlin partitions, we say that there are $n_O(\alpha)$ copies of the $\mcA$-orbit of $\alpha$ contained in $O$. 
\end{itemize}
\end{defn}

\begin{lemma}\label{l:refinement}
Any two $\sim_\Gamma $-partitions $\mcA$, $\mcB$ admit a common refinement. 
\end{lemma}

\begin{proof}
 For $n$ big enough, $\mcA_n$ is compatible with any atom of $\mcA$, and for any $\alpha$, $\beta$ belonging to the same $\mcA$-orbit, there exists $\gamma \in \Gamma_n$ such that $\gamma(\alpha)=\beta$. Hence in any $\mcA_n$-orbit $O$ we have $n_O(\alpha)=n_O(\beta)$ for large $n$. So any $\mcA_n$ refines $\mcA$ as long as $n$ is large enough, and it follows that $\mcA$, $\mcB$ have a common refinement.
\end{proof}


This implies in particular that, for any $\sim_\Gamma $-partition $\mcA$ and any clopen $U$, there exists a refinement of $\mcA$ which is compatible with $U$.

We seize the opportunity to note the following fact, closely related to \cite{Giordano1995}*{Lemma~6.1} and \cite{Giordano2004}*{Theorem~4.8}. This result will not be needed in our proof of the classification theorems (though we use it at the end of the paper to prove Theorem \ref{t:absorption2}).

\begin{theorem}\label{t:ample_vs_integers}
There exists a minimal homeomorphism $\varphi$ and $x \in X$ such that $\Gamma= \Gamma_x(\varphi)$.
\end{theorem}

\begin{proof}
We consider \emph{ordered} $\sim_\Gamma$-partitions, i.e.~$\Gamma$-partitions $\mcA$ where each orbit is totally ordered. Given two ordered $\Gamma$-partitions $\mcA$, $\mcB$, we say that $\mcB$ refines $\mcA$ if:
\begin{itemize}
    \item $\mcB$ refines $\mcA$ as a $\sim_\Gamma$-partition.
    \item For any $\mcB$-orbit $O$, each fragment $O_\mcA(\alpha)$ of $\mcA$-orbit contained in $O$ is an interval for the ordering of $O$, and the ordering on each $O_\mcA(\alpha)$ induced by $\mcB$ coincides with the ordering induced by $\mcA$.
\end{itemize}
Given an ordered $\sim_\Gamma$-partition and a $\mcA$-orbit $O$, we call $\textrm{base}(O)$ the minimal element of $O$, and $\textrm{top}(O)$ its maximal element; the base of $\mcA$ is the union of the bases of all $\mcA$-orbits, and we similarly define the top of $\mcA$.

If $\mcA$ is an ordered $\Gamma$-partition, and $\mcB$ is a $\Gamma$-partition which refines $\mcA$, then we may turn $\mcB$ into an ordered $\Gamma$-partition refining $\mcA$.

We fix a compatible metric on $X$. We build a sequence $\mcB_n$  of ordered $\sim_\Gamma$-partitions, and a sequence $\Delta_n$ of finite subgroups of $\Gamma$, such that (when we forget the ordering on $\mcB_n$) the sequence $(\mcB_n,\Delta_n)$ is an exhaustive sequence of finite unit systems and the diameters of the top and base of $\mcB_n$ both converge to $0$.
Assume for the moment that this is possible. Then there exists a unique homeomorphism $\varphi$ of $X$ such that, for every $n$ and every atom $\beta$ of $\mcB_n$ which does not belong to $\textrm{top}(\mcB_n)$, $\beta$ is mapped by $\varphi$ to its successor in $O(\beta)$; and $\varphi$ maps the top of $\mcB_n$ to the base of $\mcB_n$. Let $\{x\}= \bigcap_n \textrm{base}(\mcB_n)$. Then $\mcB_n$ is a sequence of Kakutani-Rokhlin partitions for the minimal homeomorphism $\varphi$, and $\Gamma_x(\varphi)$ has the same orbits as $\Gamma$ on $\Clopen(X)$. We deduce from Krieger's theorem that $\Gamma$ is conjugate to $\Gamma_x(\varphi)$ in $\Homeo(X)$, which gives the desired result.

It remains to explain how to build $(\mcB_n,\Delta_n)$. We let $\mcB_0$ be the trivial partition and $\Delta_0=\{\mathrm{id}\}$ to initialize the construction. Assume that $\mcB_n$ has been constructed. By cutting an orbit if necessary, we may assume that $\mcB_n$ has two orbits $O,O'$ such that $\alpha= \textrm{base}(O)$ and $\beta=\textrm{top}(O')$ have diameter less than $2^{-n}$. 

Next, we find $N$ big enough that $\Gamma_N$ contains both $\Delta_n$ and $\Gamma_n$; $\mcA_N$ refines both $\mcA_n$ and $\mcB_n$ (as unordered $\Gamma$-partitions); and each $\mcA_N$-orbit contains a fragment of $O$ as well as a fragment of $O'$ (this last property holds for any large enough $N$ by minimality of $\Gamma$). Then we may turn $\mcA_N$ into an ordered $\Gamma$-partition refining $\mcB_n$ and such that each $\mcB_n$-orbit has its base contained in $\alpha$ and its top contained in $\beta$. We let $\mcB_{n+1}=\mcA_N$ (endowed with such an ordering), $\Delta_{n+1}=\Gamma_N$. This concludes the proof.
\end{proof}

We turn to a construction that is more specifically needed in our proof of the classification theorem for minimal ample groups.

\begin{defn}
Let $U,V$ be two clopen sets, and $\mcA$  be a $\sim_\Gamma $-partition compatible with $U$ and $V$. 

\begin{itemize}
\item[•] We say that a $\mcA$-orbit $O$ is \emph{$(U,V)$-balanced} if $n_O(U)=n_O(V)$.
\item[•] We say that a pair of $\mcA$ orbits $(O(\alpha),O(\beta))$ is $(U,V)$-balanced if \[n_{O(\alpha)}(U)-n_{O(\alpha)}(V)=n_{O(\beta)}(V)-n_{O(\beta)}(U) \]
and $\mu(\alpha)=\mu(\beta)$ for all $\mu \in M(\Gamma)$.
\item[•] We say that $U,V$ are \emph{$\mcA$-equivalent} if $n_O(U)=n_O(V)$ for any $\mcA$-orbit $O$.
\end{itemize}
\end{defn}

Note that two clopens $U,V$ are $\mcA$-equivalent for some $\sim_\Gamma $-partition $\mcA$ if and only if $U \sim_\Gamma  V$.
If $U \sim_\Gamma^* V$ but $U \not \sim_\Gamma  V$, there cannot exist a $\sim_\Gamma $ partition $\mcA$ such that $U,V$ are $\mcA$-equivalent; it is however convenient to manipulate $\sim_\Gamma $-partitions which are as close as possible to making $U,V$ $\mcA$-equivalent. This is what leads us to consider balanced pairs of orbits, and motivates the next definition and lemma.

\begin{defn}
Let $\overline{ U}=(U_1,\ldots,U_n)$, $\overline{ V}=(V_1,\ldots,V_n)$ be tuples of clopen sets. Let $\mcA$ be a $\sim_\Gamma $-partition. We say that $\overline{ U }$ and $\overline{ V}$ are \emph{almost $\mcA$-equivalent} if $\mcA$ is compatible with each $U_i, 
V_i$ and there exists $k \le n$ and $\mcA$-orbits $C_1,\ldots, C_k$, $D_1,\ldots,D_k$, such that:
\begin{itemize}
\item[•] Any $\mcA$-orbit which does not belong to $\{C_1,\ldots,C_k,D_1,\ldots,D_k\}$ is $(U_j,V_j)$-balanced for all $j$.
\item[•] For any $i \in \{1,\ldots,k\}$ and any $j \in \{1,\ldots,n\}$, $(C_i,D_i)$ is a $(U_j,V_j)$-balanced pair of orbits.
\end{itemize}
We call $C_1,\ldots,C_k$, $D_1,\ldots,D_k$ the \emph{exceptional} $\mcA$-orbits.
\end{defn}

In the above definition, we allow $k=0$, in which case $\overline{ U}, \overline{ V}$ are $\mcA$-equivalent. 

\begin{prop}\label{l:almost_equivalent_partitions}
Let $\mcA$ be a $\sim_\Gamma $-partition, and $\overline{ U}=(U_1,\ldots,U_n)$, $\overline{ V}=(V_1,\ldots,V_n)$ be two tuples of clopen sets such that $U_i \sim_\Gamma ^* V_i$ for all $i \in \{1,\ldots,n\}$.

Then there exists a $\sim_\Gamma $-partition $\mcB$ which refines $\mcA$ and is such that $(\overline{ U}, \overline{ V})$ are almost $\mcB$-equivalent.

Additionally, one can ensure that the following condition holds: denote 
$$N_i(\mcB)= \max \{|n_O(U_i) - n_O(V_i)| \colon O \text{ is a } \mcB\text{-orbit}\}$$
Let $N(\mcB)= \sum_{i=1}^n N_i(\mcB)$ and denote by $h$ be the number of atoms of $\mcA$. Then any exceptional $\mcB$-orbit contains more than $3h N(\mcB))$ copies of every $\mcA$-orbit.
\end{prop}

The proof of the proposition is based on an argument initially used in the proof of \cite{Ibarlucia2016}*{Proposition~3.5}; it is the key combinatorial step of our argument. The proof is elementary and based on repeated applications of Lemma \ref{l:ample_Glasner_Weiss}, but somewhat tedious. The last part of the statement plays a technical part in the proof of the classification theorem for minimal ample groups (in the language used in the proof of that theorem , it is used to ensure that the $\tilde \Gamma$-orbits of singular points are distinct) and can safely be ignored on first reading.

We prove the first part of the statement (before \enquote{additionally...}) by induction on $n$. We treat the case $n=1$ separately.
\begin{proof}[The case $n=1$]
Fix $\mcA$ and $U,V$ such that $U \sim_\Gamma^* V$. 

There exists a $\sim_\Gamma$ partition $\mcB$ which refines $\mcA$ and is compatible with $U,V$; for such partitions we simply denote 
$$N(\mcB)=\max \{|n_O(U) - n_O(V)| \colon O \text{ is a } \mcB\text{-orbit}\}$$
Consider a partition $\mcB$ refining $\mcA$ and for which $N(\mcB)$ is minimal (denote it $N$ from now on); we assume that $N \ge 1$, otherwise  $(U,V)$ are $\mcB$-equivalent and we have nothing to do.

Let $O(\alpha_1),\ldots,O(\alpha_p) $ enumerate the $\mcB$-orbits for which $n_O(U)-n_O(V)=N$ (if any such orbit exists). Let also $O(\beta_1),\ldots,O(\beta_q)$ enumerate the orbits such that $n_O(V) >n_O(U)$. 

We have, for all $\mu \in M(\Gamma)$,
\begin{align*}
\mu(U)- \mu(V) & \ge  N  \mu \left( \bigsqcup_{i=1}^p \alpha_i \right) - \sum_{j=1}^q \left( n_{O(\beta_j)}(V)- n_{O(\beta_j)}(U) \right)\mu(\beta_j) \\
               & \ge   N  \mu \left( \bigsqcup_{i=1}^p \alpha_i \right) - N  \mu \left( \bigsqcup_{j=1}^q \beta_j \right)  
\end{align*}
Since $\mu(U)=\mu(V)$ for all $\mu \in M(\Gamma)$, we must have that 
$$\forall \mu \in M(\Gamma) \quad \mu \left( \bigsqcup_{i=1}^p \alpha_i \right) \le \mu \left( \bigsqcup_{j=1}^q \beta_j \right)  $$
and there can be equality only if  $|n_O(V)-n_O(U)|=N$ for any orbit such that $n_O(V)\ne n_O(U)$. If we are not in this situation, we may apply Lemma \ref{l:ample_Glasner_Weiss} to find an element $\gamma$ of $\Gamma$ mapping $\bigsqcup_{i=1}^p \alpha_i$ into $\bigsqcup_{j=1}^q \beta_j$; from this, we can produce a $\sim_\Gamma $-partition $\mcC$ refining $\mcB$, and such that there does not exist any $\mcC$-orbit for which $n_O(U)-n_O(V)=N$. 
Let us detail a bit here how $\mcC$ is produced; this explanation is intended for readers who are not familiar with the cutting-and-stacking procedure and every partition produced in the remainder of this proof will be obtained similarly. Take the $\sim_\Gamma$-partition $\mcB'$ obtained by cutting $\mcB$ (as in Lemma \ref{l: making a KR partition compatible with a clopen}) to refine the $\gamma(\alpha_i)$'s and the $\gamma^{-1}(\beta_j)$'s. Call, for all $i\le p, j\le q$,
$$\alpha_{i,j}=\alpha_i\cap \gamma^{-1}(\beta_j) \text{ and }  \beta_{i,j}=\beta_j\cap \gamma(\alpha_i)$$
Clearly, for all $i\le p, j\le q$, one has that $\beta_{i,j}=\gamma(\alpha_{i,j})$. We now form a new $\sim_\Gamma$-partition by joining the $\mcB'$-orbit of $\alpha_{i,j}$ and $\beta_{i,j}$, for each $i,j$, and leaving the other orbits unchanged (this is the analogue of what we call stacking for Kakutani--Rokhlin partitions). Denote by $\mcC$ this new $\sim_\Gamma$-partition.

Using the same argument with the roles of $U$ and $V$ reversed, we see that we can build a $\sim_\Gamma $-partition $\mcD$ refining $\mcC$ and such that $N(\mcD)< N$, unless any orbit $O$ of $\mcC$ which is not $(U,V)$-balanced is such that $|n_O(U)-n_O(V)|=N$. By definition of $N$, we must thus be in that particular case. 

Let again $O(\alpha_1),\ldots,O(\alpha_p)$ enumerate all $\mcC$-orbits $O$ with $n_O(U) > n_O(V)$, and $O(\beta_1),\ldots,O(\beta_q)$ enumerate those with $n_O(V)>n_O(U)$. Assume that $p \ge 2$. There exists $\gamma \in \Gamma$ mapping $\bigsqcup_{i=2}^p \alpha_i$ into $\bigsqcup_{j=1}^q \beta_j$. By forming appropriate clopen partitions of $\bigsqcup_{i=2}^p \alpha_i$, $\bigsqcup_{j=1}^q \beta_j$ and matching them together (similarly to how we defined $\mcC$ above), we produce a $\sim_\Gamma $-partition with only one orbit $O$ for which $n_O(U)>n_O(V)$. Applying the same argument to this new partition, with the roles of $U$ and $V$ reversed, we obtain a $\sim_\Gamma $-partition with exactly one orbit $O$ for which $n_O(U)-n_O(V)>0$, and one orbit $P$ for which $n_P(V)-n_P(U)>0$. If $O=O(\alpha)$, $P=O(\beta)$, we must have $\mu(\alpha)=\mu(\beta)$ for all $\mu \in M(\Gamma)$, since 
$$\forall \mu \in M(\Gamma) \quad 0=\mu(U)-\mu(V)= N \mu(\alpha)- N \mu(\beta) $$

This proves the first part of the Proposition's statement in case $n=1$.
\end{proof}

Now we assume that the first part of the statement of Proposition \ref{l:almost_equivalent_partitions} has been established for some $n$, and need to explain why it is true for $n+1$. 

We fix $\mcA$, $(U_1,\ldots,U_{n+1})$ and $(V_1,\ldots,V_{n+1})$ such that $U_i \sim_\Gamma^* V_i$ for all $i \in \{1,\ldots,n+1\}$. To simplify notation below we let $U=U_{n+1}$ and $V=V_{n+1}$.
We consider partitions $\mcB$ refining $\mcA$ and such that $(U_1,\ldots,U_n)$, $(V_1,\ldots,V_n)$ are almost $\mcB$-equivalent (these exist by our induction hypothesis), as 
witnessed by exceptional orbits $O(\alpha_1^\mcB),\ldots,O(\alpha_k^\mcB)$, $O(\beta_1^\mcB),\ldots,O(\beta_k^\mcB)$ for some $k \le n$. 

For such a $\mcB$, We denote 
$$Y^\mcB= X \setminus \bigsqcup_{i=1}^k (O(\alpha_i^\mcB) \sqcup O(\beta_i^\mcB))$$

\begin{lemma}\label{l:inductivestep}
Assume that there exists $\mcB$ refining $\mcA$, such that $(U_1,\ldots,U_n)$ and $(V_1,\ldots,V_n)$ are almost $\mcB$-equivalent and each pair $(O(\alpha_i^\mcB),O(\beta_i^\mcB))$ is $(U,V)$-balanced. Then there exists $\mcC$ refining $\mcA$ and such that $(U_1,\ldots,U_{n+1})$ and $(V_1,\ldots,V_{n+1})$ are almost $\mcC$-equivalent.
\end{lemma}

\begin{proof}
We may assume that $Y^\mcB$ is nonempty: if it is empty, take $\tau \subsetneq \alpha_1^\mcB$, then find $\pi \subsetneq \beta_1^\mcB$ such that $\tau \sim_\Gamma \pi$ (Lemma \ref{l:ample_Glasner_Weiss} shows that this is possible). Then form a partition $\mcC$ refining $\mcB$ by stacking together the $\mcB$-orbits of $\tau$ and $\pi$, and having $O(\alpha_1^\mcB \setminus \tau)$, $O(\beta_1^\mcB \setminus \pi)$ form an exceptional column pair. This partition is still such that $(U_1,\ldots,V_n)$ and $(V_1,\ldots,V_n)$ are $\mcC$-equivalent, and $Y^\mcC$ is nonempty since it contains $\tau$.

Assuming that $Y^\mcB$ is nonempty, note that $U \cap Y^\mcB \sim_\Gamma^* V \cap Y^\mcB$; we simply repeat the argument of the proof of the case $n=1$ of Proposition \ref{l:almost_equivalent_partitions}, working inside $Y^\mcB$. The point is that for each $i$ the construction joins together fragments of orbits which are $(U_i,V_i)$-balanced for all $i \le n$, and these new orbits are still $(U_i,V_i)$-balanced for each $i \le n$.

This produces the desired $\mcC$; every $\mcC$-orbit outside $Y^\mcB$ coincides with a $\mcB$-orbit. This adds at most one exceptional pair of orbits to the exceptional pairs of orbits of $\mcB$ (which remain exceptional for $\mcC$).
\end{proof}



\begin{proof}[Inductive step of the proof of Proposition \ref{l:almost_equivalent_partitions}]
For a $\sim_\Gamma$-partition $\mcB$ refining $\mcA$ and for which $(U_1,\ldots,U_n)$, $(V_1,\ldots,V_n)$ are almost $\mcB$-equivalent, we denote 
$$M(\mcB)=\max_{1\le i\le k}|n_{O(\alpha_i^\mcB)}(U)+n_{O(\beta_i^\mcB)}(U) - n_{O(\alpha_i^\mcB)}(V)-n_{O(\beta_i^\mcB)}(V)|$$
Take $\mcB$ so that $M(\mcB)$ is minimal (denote it $M$ from now on). Our goal is to prove that $M=0$, for then we obtain the desired result by applying Lemma \ref{l:inductivestep}.

A key point to note here is that, if we join together a fragment of some $O(\alpha_i^\mcB)$ with a fragment of $O(\beta_i ^\mcB)$ (using Lemma \ref{l:ample_Glasner_Weiss} as we did earlier) to form a finer partition $\mcC$, this finer partition satisfies the same conditions as $\mcB$, and the new orbit obtained after this joining is contained in $Y^\mcC$.

Now, assume for a contradiction that $M \ge 1$. Up to reordering (and exchanging the roles of $U$ and $V$), we suppose that 
$$n_{O(\alpha_1^\mcB)}(U)+n_{O(\beta_1^\mcB)}(U) - n_{O(\alpha_1^\mcB)}(V)-n_{O(\beta_1^\mcB)}(V)=M$$
If there is no orbit $O$ in $Y_\mcB$ such that $n_O(V)-n_O(U)=K>0$, there must exist $1\le j\le k$ such that 
$$n_{O(\alpha_j^\mcB)}(V)+n_{O(\beta_j^\mcB)}(V) - n_{O(\alpha_j^\mcB)}(U)-n_{O(\beta_j^\mcB)}(U)>0$$
As we already mentioned above, if we join together (using $\Gamma$) a fragment of $O(\alpha_j^\mcB)$ with a fragment of $O(\beta_j^\mcB)$, we create a partition $\mcC$ refining $\mcB$, satisfying all the relevant conditions, and such that an orbit $O$ with $n_O(V)-n_O(U)=K$ exists in $Y^\mcC$; so we may as well assume that such an orbit exists in $Y^\mcB$.

Let $K=qM+r$ be the euclidean division of $K$ by $M$. If $q>0$, we can join together a small fragment of $O(\alpha_1^\mcB)$, one of $O(\beta_1^\mcB)$ and one of $O$ to create a finer partition $\mcB'$ along with an orbit $O'$ contained in $Y_{\mcB'}$ and such that $n_{O'}(V)-n_{O'}(U)=K-M$. Repeating this $q-1$ times, we obtain a finer partition $\mcB''$ along with an orbit $O(\delta)$ inside $Y^{\mcB''}$ such that 
\[0 < n_{O(\delta)}(V)-n_{O(\delta)}(U)=M+r< 2M \]
If $q=0$ then choose $\delta$ such that $O=O(\delta)$; the above inequality is also satisfied.
Again to simplify notation, we assume that such an orbit already exists in $\mcB$ (since we just reduced to that case).

Choose ${\alpha'_1}$ inside $\alpha_1^ \mcB$ such that $0< \mu({\alpha'_1}) < \mu(\delta)$ for all $\mu \in M(\Gamma)$.

Then find $\gamma \in \Gamma$ such that $\gamma({\alpha'_1}) \subset \beta_1^\mcB$, and set ${\beta'_1}=\gamma({\alpha'_1}^\mcB)$.
Cut the orbit of $\alpha_ 1^\mcB$ to form two orbits $O({\alpha'_1})$ and $O(\alpha_1^\mcB \setminus {\alpha'_1})$; do the same to the orbit of $\beta_1^\mcB$.
Then join together the orbit of $\alpha_1^\mcB \setminus {\alpha'_1}$ and the orbit of $\beta_1^\mcB \setminus {\beta'_1}$ (thus producing an orbit which is $(U_i,V_i)$-balanced for all $i \in \{1,\ldots,n\})$; and the orbit of ${\alpha'_1}$ with a fragment of $O(\delta)$, which is possible because $\mu({\alpha'_1}) < \mu(\delta)$ for all $\mu \in M(\Gamma)$.

We obtain a $\sim_\Gamma$-partition $\mcC$ which refines $\mcB$, and such that $(U_1,\ldots,U_n)$ and $(V_1,\ldots,V_n)$ are almost $\mcC$-equivalent with exceptional $\mcC$-orbits of the form 
\[O({\alpha'_1}),O(\alpha_2^\mcB),\ldots,O(\alpha_k^\mcB),O({\beta'_1}),\ldots,O(\beta_k^\mcB)\]
Further, 
\[|n_{O({\alpha'_1})}(U)+n_{O({\beta'_1})}(U) - n_{O({\alpha'_1})}(V)-n_{O({\beta'_1})}(V)|<M\]
Applying this argument repeatedly, we conclude that there exists a $\sim_\Gamma$-partition $\mcD$ refining $\mcB$ and such that $M(\mcD)< M(\mcB)$. This is a contradiction, so we conclude that $M=0$ as promised.
\end{proof}

\begin{proof}[End of the proof of Proposition \ref{l:almost_equivalent_partitions}]
Let us now see why the \enquote{additionally} part holds true. 
We first pick a $\sim_\Gamma$-partition $\mcB$ satisfying the 
conditions in the first part of the lemma, and denote as before $O(\alpha_1^\mcB),\ldots,O(\alpha_k^\mcB)$, $O(\beta_1^\mcB),\ldots,O(\beta_k^\mcB)$ the exceptional $\mcB$-orbits and
$$Y^\mcB=X \setminus \bigsqcup_{i=1}^k (O(\alpha_i^\mcB) \cup O(\beta_i^\mcB))$$
By joining together fragments of some $O(\alpha_i^\mcB)$ and $O(\beta_i^\mcB)$ if necessary, we can guarantee that for every $\mcA$-orbit $O$ there is a fragment of $O$ contained in $Y^\mcB$. Then, by joining together some sufficiently small fragments of orbits contained in $Y^\mcB$, we ensure that in $Y^\mcB$ there exists an orbit $O(\delta)$ which contains more than $3hN(\mcB))$ copies of each $\mcA$-orbit. 

Again by joining together if necessary some fragments of $O(\alpha_i^\mcB)$ and $O(\beta_i^\mcB)$, we can assume that $2n \mu(\alpha_i^\mcB)  < \mu(\delta)$ for all $\mu \in M(\Gamma)$ and all $i$. Then we may join for each $i$ a fragment of $O(\delta)$ and a fragment of $O(\alpha_i^\mcB)$, as well as a fragment of $O(\delta)$ and a fragment of $O(\beta_i^\mcB)$. Each orbit in these new exceptional orbit pairs contains many copies of each $\mcA$-orbit.
\end{proof}

The fact that $k \le n$ in the lemma above does not play a part in our arguments; but it is important to have some control over the exceptional orbits. 

In the remainder of the paper, we often identify a clopen partition $\mcA$ with the Boolean algebra it generates (and start doing this in the lemma below).

\begin{lemma}\label{l:from_partitions_to_an_ample_group}
Let $\simeq$ be a full equivalence relation on $\Clopen(X)$, and $(\mcB_n)$ be a sequence of $\simeq$-partitions such that
\begin{itemize}
\item For all $n$, $\mcB_{n+1}$ refines $\mcB_n$.
\item For each clopen $U$, there exists $n$ such that $\mcB_n$ is compatible with $U$.
\end{itemize} 

Then there exists an ample group $\Lambda=\bigcup_n \Lambda_n$ such that 
\begin{itemize}
\item For all $n$, $(\mcB_n,\Lambda_n)$ is a unit system.
\item For all $n$, all $\alpha \in \mcB_n$, the $\mcB_n$-orbit of $\alpha$ coincides with $\Lambda_n \alpha$.
\item For any clopen $A$ and any $\lambda \in \Lambda$ one has $A \simeq \lambda(A)$.  
\end{itemize}
\end{lemma}

\begin{proof}
Denote by $\mcO_n$ the set of $\mcB_n$-orbits, enumerate it as $(O_1^n,\ldots,O_{k_n}^n)$, and for $i \in \{1,\ldots,k_n\}$ denote by $h^n_i$ the cardinality of $O^n_i$. Each group $\Lambda_n$ will be isomorphic, as an abstract group, to $\prod_{i=1}^{k_n} \sym_{\mcO^n_i}$.

Let $\mcB_{-1}$ be the trivial partition, and $\Lambda_{-1}=\{\mathrm{id}\}$ to initialize the process; then assume that $\Lambda_n$ has been constructed. For every $i$, denote by $(A_{i,j})_{j< h^n_i}$ an enumeration of $\mcO^n_i$, and by $\tau^i_j$ the transposition $(A_{i,j},A_{i,j+1})$ of $\sym_{\mcO^n_i}$ for $j< h^n_i-1$. We first use the fact that $\mcB_{n+1}$ refines $\mcB_n$ to extend each $\tau^i_j$ to a permutation of the atoms of each $\mcO_k^{n+1}$; to see how this is done, fix some $\tau=\tau^i_j$, and let $\alpha,\beta \in \mcO^{n}_i$ be the two only atoms that are not fixed by $\tau$ . For any $\mcB_{n+1}$-orbit $\mcO^{n+1}_k$, there are as many atoms $\alpha^k_1,\ldots,\alpha^k_l$ contained in $\alpha$ as there are atoms $\beta^k_1,\ldots,\beta^k_l$ contained in $\beta$, because $\mcB_{n+1}$ refines $\mcB_n$. So we can extend $\tau$ to a permutation of the atoms of $\mcB_{n+1}$, by setting $\tau(\alpha^k_m)=\beta^k_m$, $\tau(\beta^k_m)=\alpha^k_m$ for every $k,m$ (we note here that our extension depends on the choice of enumerations of the atoms $\alpha_l^k$, $\beta_l^k$; this  freedom will be used during the proof of the classification theorem for minimal ample groups).

Once this is done for every $\tau_i^j$, we see $\Lambda_n$ as a subgroup of the permutation group  $\Lambda_{n+1}=\prod_{i=1}^{k_{n+1}} \sym_{\mcO^{n+1}_i}$ and can move on to the next step.

Since any clopen set eventually belongs to some $\mcB_n$, we can view each $\lambda \in \Lambda$ as an automorphism of $\Clopen(X)$, that is, a homeomorphism of $X$. 

By construction, $\Lambda$ is then an ample subgroup of $\Homeo(X)$, and the $\mcB_n$-orbit of any $\alpha \in \mcB_n$ coincides with $\Lambda_n \alpha$.

To check the last point, fix a clopen $A$ and $\lambda \in \Lambda$. There exists $n$ such that $A$ is a union of atoms $\alpha_1,\ldots, \alpha_n$ of $\mcB_n$, and $\lambda$ belongs to $\Lambda_n$. Then 
$$\lambda(A) = \bigsqcup \lambda(\alpha_n) $$
and for all $n$, we have that $\lambda(\alpha_n) \simeq \alpha_n$ since $\mcB_n$ is a $\simeq$-partition. Since $\simeq$ is full, this establishes the final condition of the lemma.
\end{proof}

\section{The classification theorem for minimal ample groups}
\subsection{An absorption theorem}
We now prove an ``absorption theorem,'' which is a particular case of \cite{Giordano2004}*{Theorem~4.6}; see also the further generalizations in \cite{Giordano2008}*{Lemma~4.15} and \cite{Matui2008}*{Theorem~3.2}. 

\begin{theorem}\label{t:absorption}
Let $\Gamma$ be an ample subgroup acting minimally. Let $K= K_1 \sqcup \sigma(K_1)$ be a closed subset of $X$ without isolated points, where $K_1$ is closed and $\sigma \in \Homeo(K)$ is an involution. Assume that $K$ is $\Gamma$-sparse.

Denote by $R_{\Gamma,K}$ the finest equivalence relation coarser than $R_\Gamma$ and for which $k,\sigma(k)$ are equivalent for all $k \in K$. Then there exists an ample group $\Sigma$ whose action induces $R_{\Gamma,K}$; furthermore $\Sigma$ and $\Gamma$ are orbit equivalent.
\end{theorem}

At the end of the paper, we explain how one can remove the assumption that $K$ has no isolated points from the above statement.

\begin{proof}
The strategy of proof is as follows: we show first that there exists a sparse set $S_1 \sqcup \pi(S_1)$ without isolated points (with $\pi$ a homeomorphic involution), such that $R_{\Gamma,S}$ is induced by an ample group which is orbit equivalent to $\Gamma$. Then we use Theorem \ref{t:Krieger_pointed} to conclude that this result holds for $R_{\Gamma,K}$.

First, we fix a refining  sequence $(\mcA_n,\Gamma_n)$ of finite unit systems such that $\Gamma = \bigcup_n \Gamma_n$, $\Clopen(X)=\bigcup_n \mcA_n$, with the additional property that for all $n$ there exists two disjoint $\mcA_n$-orbits $O_n(\alpha_n)$, $O_n(\beta_n)$ (which we call the \emph{exceptional} orbits) such that:
\begin{itemize}
\item For all $n$, $\alpha_n$ and $\beta_n$ are $\mcA_{n+1}$-equivalent.
\item For all $n$, denote by $h_n$ the cardinality of $\mcA_n$. Then $O_{n+1}(\alpha_{n+1})$ contains at least $2 h_n$ copies of $O_n(\alpha_n)$, and $O_{n+1}(\beta_{n+1})$ contains at least $2h_n$ copies of $O_n(\beta_n)$.
\end{itemize}
It is straightforward to build such a sequence using the same techniques (cutting and stacking, and Lemma \ref{l:ample_Glasner_Weiss}) as in the previous section, so we do not give details here.
We denote by $\tilde \mcA_n$ the refinement of $\mcA_n$ obtained by joining together $O(\alpha_n)$ and $O(\beta_n)$ (the corresponding orbit is called the exceptional orbit of $\tilde \mcA_n$). Since each orbit of $\mcA_{n+1}$ is $(\alpha_n,\beta_n)$-balanced, $\tilde \mcA_{n+1}$ refines $\tilde \mcA_n$. 

We define inductively a homeomorphic involution $\pi$ with the following properties:
\begin{itemize}
    \item For all $n$, $\pi$ induces an automorphism of $\tilde \mcA_n$, which maps each non-exceptional orbit to itself.
    \item Denoting by $\Lambda_n$ the subgroup of $\Aut(\tilde \mcA_n)$ generated by $\Gamma_n$ and $\pi$, $(\tilde \mcA_n,\Lambda_n)$ is a unit system and the $\Lambda_n$-orbit of any atom of $\tilde \mcA_n$ coincides with its $\tilde \mcA_n$-orbit.
    \item Say that an atom $\alpha$ of $\mcA_n$ is \emph{singular} if $\pi$ does not coincide with an element of $\Gamma_n$ on $\alpha$ ($\alpha$ must belong to one of the two exceptional orbits, and be mapped by $\pi$ to the other exceptional $\mcA_n$-orbit). Then for every atom $\alpha$ of $\mcA_{n+1}$ there exists at most one singular atom in $\Gamma_n \alpha$; and for every singular atom $\alpha$ of $\mcA_n$ there exist at least two singular atoms of $\mcA_{n+1}$ contained in $\alpha$.
\end{itemize}
Initialize the construction by setting $\pi(\alpha_0)=\beta_0$, $\sigma(\beta_0)=\alpha_0$, and $\pi$ is the identity on the other atoms of $\mcA_0$.

Next, assume that we have extended $\pi$ to an automorphism of $\tilde \mcA_n$ satisfying the various conditions above. In particular, on all nonsingular atoms of $\mcA_n$, we have already declared $\pi$ to be equal to some element of $\Gamma_n$, hence the extension of $\pi$ to atoms of $\tilde \mcA_{n+1}$ contained in a nonsingular atom is already defined. Since each $\mcA_n$-orbit is $(\alpha,\pi(\alpha))$-balanced for every singular atom of $\mcA_n$, we can also extend $\pi$ to all nonexceptional $\mcA_{n+1}$-orbits so that it coincides with an element of $\Gamma_{n+1}$ on these orbits.

It remains to explain how to extend $\pi$ to the two exceptional orbits $O_{n+1}(\alpha_{n+1})$, $O_{n+1}(\beta_{n+1})$. Let $\tau_1,\ldots,\tau_p$ be singular atoms of $\mcA_n$ which belong to distinct $\pi$-orbits and such that for any singular atom $\alpha$ of $\mcA_n$ there exists $i$ such that $\alpha=\tau_i$ or $\alpha= \pi(\tau_i)$. For all $i$, we set aside two distinct atoms $\theta^i_1$, $\theta^i_2$ of $O_{n+1}(\alpha_{n+1})$ contained in $\tau_i$, and two distinct atoms $\delta^i_1$, $\delta^i_2$ of $O_{n+1}(\beta_{n+1})$ contained in $\pi(\tau_i)$. We do this while ensuring that no two of these atoms belong to the same $\Gamma_n$-orbit (which is possible since there are many copies of $O_n(\alpha_{n})$ in $O_{n+1}(\alpha_{n+1})$, and many copies of $O_n(\beta_{n})$ in $O_{n+1}(\beta_{n+1})$). Next we set $\pi(\theta^i_j)=\delta^i_j$, $\pi(\delta^i_j)=\theta^i_j$.

For any $i \in \{1,\ldots,p\}$, there remain as many atoms of $O_{n+1}(\alpha_{n+1})$ contained in $\tau_i$ on which $\pi$ is yet to be defined as there are such atoms in $\pi(\tau_i)$; so we can extend $\pi$ so that it coincides with an element of $\Gamma_{n+1}$ on those atoms. We do the same in $O_{n+1}(\beta_{n+1})$ to finish extending $\pi$ to an involutive automorphism of $\tilde \mcA_{n+1}$.

Say that $x \in X$ is \emph{singular} if the atom $\alpha_n(x)$ of $\mcA_n$ which contains $x$ is singular for all $n$. Since any singular atom of $\mcA_n$ contains at least two singular atoms of $\mcA_{n+1}$, we see that the set $S$ of singular points does not have any isolated point. Also, for every singular point $x$, $\alpha_{n+1}(x)$ is the unique singular atom contained in $\Gamma_n \alpha_{n+1}(x)$, which means that $x$ is the unique singular point in $\Gamma_n x$. Thus $S$ is $\Gamma$-sparse (clearly $S$ is closed since it is an intersection of clopen sets).

Let $\Lambda$ be the full group generated by $\Gamma$ and $\pi$. It is an ample group since $(\tilde \mcA_n,\Lambda_n)$ is a sequence of finite unit systems and $\Lambda=\bigcup_n \Lambda_n$. $R_\Lambda$ is obtained from $R_\Gamma$ by gluing together the $\Gamma$-orbits of $x$ and $\pi(x)$ for all $x \in S$; in other words, it is the finest equivalence relation coarser than $R_\Gamma$ and such that $x,\pi(x)$ are equivalent for all $x \in S$.

By construction, the actions of $\Lambda$ and $\Gamma$ have the same orbits when they act on $\Clopen(X)$. Hence $\Lambda$ and $\Gamma$ are orbit equivalent by Krieger's theorem.

We can find a closed subset $S_1$ of $S$ such that $S=S_1 \sqcup \pi(S_1)$.
Let $h \colon S_1 \to K_1$ be any homeomorphism (both $K_1$ and $S_1$ are Cantor sets), and extend $h$ to a homeomorphism from $S$ to $K$ by setting $h(\pi(x)= \sigma(h(x))$ for all $x \in S_1$.

By Theorem \ref{t:Krieger_pointed}, there exists $g \in \Homeo(X)$ such that $g\Gamma g^{-1}= \Gamma$, and $g_{|S}= h$. Then $\Sigma= g \Lambda g^{-1}$ is an ample group which induces $R_{\Gamma,K}$. This ample group is conjugate to $\Lambda$, hence orbit equivalent to $\Gamma$.
\end{proof}

\subsection{Proof of the classification theorem for minimal ample groups}
We are now ready to prove the key result leading to the classification of minimal $\Z$-actions on the Cantor space.

\begin{theorem}[The classification theorem for minimal ample groups]\label{t:minimal_ample}
Let $X$ be the Cantor space. Given two ample subgroups of $\Homeo(X)$ acting minimally, the following conditions are equivalent.
\begin{itemize}
\item[•] The actions of $\Gamma$ and $\Lambda$ are orbit equivalent.
\item[•] There exists a homeomorphism $g$ of $X$ such that $g_*M(\Gamma)=M(\Lambda)$.
\end{itemize}
\end{theorem}

We fix an ample group $\Gamma$ acting minimally on $X$, and reuse the same notations as in the previous section.

We saw earlier (cf.~Theorem \ref{t:Krieger_gives_saturated_GPS}) that Krieger's theorem gives a proof of the classification theorem for saturated minimal ample groups. Thus our work consists in proving that $\Gamma$ is orbit equivalent to a saturated, ample subgroup $\Lambda$ of $\Homeo(X)$. 

The idea of the proof is as follows: using ideas similar to those of the previous section (and Proposition \ref{l:almost_equivalent_partitions}) we build a Cantor set $K$, and a homeomorphic involution $\pi$ such that $K \cap \pi(K)= \emptyset$, $K \cup \pi(K)$ is $\Gamma$-sparse, and the equivalence relation induced by gluing together the $\Gamma$-orbits of $x$ and $\pi(x)$ for every $x \in K$ is induced by an action of a saturated minimal ample group $\Lambda$. The absorption theorem \ref{t:absorption} yields that $\Gamma$ is orbit equivalent to $\Lambda$, from which we obtain as desired that $\Gamma$ is orbit equivalent to a saturated minimal ample group.




We now begin the proof. Let $(U_n,V_n)$ be an enumeration of all pairs of $\sim_\Gamma^*$-equivalent clopens, and assume for notational simplicity that $U_0$, $V_0$ are disjoint and $U_0 \not \sim_\Gamma V_0$.

\begin{lemma}\label{l:main_construction}
We may build a sequence of $\sim_\Gamma$-partitions $(\mcA_n)$, with distinguished orbit pairs $O(\alpha_1^n),\ldots,O(\alpha_{k_n}^n), O(\beta_1^n), \ldots O(\beta_{k_n}^n)$ $(k_n \ge 1$ for all $n$) satisfying the following conditions. 

\begin{enumerate}
\item $k_0=1$, $\alpha_1^0=U_0$, $\beta_1^0=V_0$ and $\mcA_0 = \{\alpha_1^0,\beta_1^0,X \setminus (\alpha_1^0 \cup \beta_1^0)\}$ (three orbits of cardinality $1$).

\noindent For all $n$ one has: 
\item $\mcA_{n+1}$ refines $\mcA_n$.
\item If $U_n \sim_\Gamma V_n$ then $U_n$ and $V_n$ are $\mcA_n$-equivalent.
\item The tuples  $(\alpha_1^n,\ldots,\alpha_{k_n}^n, U_{n+1})$ and $(\beta_1^n,\ldots,\beta_{k_n}^n,V_{n+1})$ are almost $\mcA_{n+1}$-equivalent, as witnessed by the exceptional orbits 
$$O(\alpha_1^{n+1}),\ldots,O(\alpha_{k_{n+1}}^{n+1}) \ , \ O(\beta_1^{n+1}),\ldots,O(\beta_{k_{n+1}}^{n+1})$$
\item For all $i,j$ $\alpha_i^j \not \sim_\Gamma \beta_i^j$.
\item Let $h_n$ be the number of atoms of $\mcA_n$; denote 
\begin{eqnarray*}
N^n_i&=& \max\{|n_O(\alpha^n_i) - n_O(\beta^n_i)| \colon O \text{ is a } \mcA_{n+1}-\text{orbit} \} \quad (i \le k_n) \\
N^{(n)} &=& \sum_{i=1}^{k_n} N^n_i
\end{eqnarray*}
Then every exceptional $\mcA_{n+1}$-orbit contains more than $3h_n N^{(n)}$ copies of every $\mcA_n$-orbit.
\end{enumerate}
\end{lemma}

\begin{proof}
Assume the construction has been carried out up to some $n$ (the case $n=0$ being dealt with by the first condition above). 

If $U_{n+1} \sim_\Gamma V_{n+1}$, find a $\sim_\Gamma$-partition $\mcB$ refining $\tilde \mcA_n$ and such that $U_{n+1}$ and $V_{n+1}$ are $\mcB$-equivalent; such a partition exists because there exists a $\sim_\Gamma$-partition for which $U_{n+1}$ and $V_{n+1}$ are $\mcB$-equivalent, and any two $\sim_\Gamma$-partitions have a common refinement. Then apply Proposition \ref{l:almost_equivalent_partitions} to this partition and $(\alpha_1^n,\ldots,\alpha_{k_n}^n)$ and $(\beta_1^n,\ldots,\beta_{k_n}^n)$. Choose $\mcA_{n+1}$ with a minimal number of exceptional columns, which ensures that no $\alpha^{n+1}_j$ and $\beta^{n+1}_j$ belong to the same $\Gamma$-orbit.

If $U_{n+1} \not \sim_\Gamma V_{n+1}$, apply Proposition \ref{l:almost_equivalent_partitions} (again, with a minimal number of exceptional columns) to $\mcA_n$, $(\alpha_1^n,\ldots,\alpha_{k_n}^n,U_{n+1})$ and $(\beta_1^n,\ldots,\beta_{k_n}^n,V_{n+1})$. 
\end{proof}

We obtain a sequence of $\sim_\Gamma^*$-partitions $(\mcB_n)$ by joining together the $\mcA_n$-orbits of each $\alpha_i^n$ and $\beta_i^n$ (i.e.~every pair of exceptional orbits of $\mcA_n$ are joined together so as to form a single $\mcB_n$-orbit consisting of two $\mcA_n$-orbits; the other orbits are unchanged). The construction ensures that, for all $n$:
\begin{itemize}
\item $\mcB_{n+1}$ refines $\mcB_n$. Indeed, for all $n$ and $i \le k_n$, $\alpha_i^n$ and $\beta_i^n$ are $\mcB_{n+1}$-equivalent (since they are almost $\mcA_{n+1}$-equivalent, and we have joined together each exceptional pair of orbits of $\mcA_{n+1}$ to form a single $\mcB_{n+1}$-orbit). And any two atoms belonging to the same non-exceptional $\mcA_n$-orbit are $\mcA_{n+1}$-equivalent, hence $\mcB_{n+1}$-equivalent.
\item For any $n$ $U_n$ and $V_n$ are $\mcB_n$-equivalent: If $U_n \sim_\Gamma V_n$ then $U_n$ and $V_n$ are already $\mcA_n$-equivalent and $\mcB_n$ refines $\mcA_n$; if $U_n \not \sim_\Gamma V_n$ then stacking the exceptional $\mcA_n$-orbits together makes them $\mcB_n$-equivalent. It follows that $U_n$, $V_n$ are $\mcB_m$-equivalent for all $m \ge n$.
\end{itemize}

\begin{lemma}
We may build:
\begin{itemize}
    \item An ample group $\tilde \Gamma = \bigcup_n \tilde \Gamma_n$ such that $(\mcA_n,\tilde \Gamma_n)$ is a unit system and the $\tilde \Gamma_n$-orbit of any atom $\alpha$ of $\mcA_n$ is equal to $\tilde \Gamma_n \alpha$.
    \item An involution $\pi$ such that for all $n$ $\pi$ induces an automorphism of $\mcB_n$, $\pi(\alpha_i^n)=\beta_i^n$ for all $n$ and all $i \in \{1,\ldots,k_n\}$, and $\pi$ is trivial outside of $\alpha_ 1^0 \sqcup \beta_1^0$. 
\end{itemize}
Denoting by $\Lambda_n$ the subgroup of $\Aut(\mcB_n)$ generated by $\tilde \Gamma_n$ and $\pi$, we ensure that $(\mcB_n,\Lambda_n)$ is a unit system and the $\Lambda_n$-orbit of any atom of $\mcB_n$ coincides with its $\mcB_n$-orbit.

We say that atoms of $\tilde \Gamma_{n}$ on which $\pi$ does not coincide with an element of $\tilde \Gamma_n$ are \emph{singular}; for every such atom $\alpha$ $\pi(\alpha)$ does not belong to $\tilde \Gamma_n \alpha$.

We also ensure that the following conditions hold:
\begin{enumerate}
    \myitem[($\ast$)]\label{sparse} For every atom $\alpha$ of $\mcA_{n+1}$ there exists at most one singular atom of $\mcA_{n+1}$ contained in $\tilde \Gamma_n \alpha$.
    \myitem[($\ast \ast$)]\label{no_isolated_points} For every singular atom $\alpha$ of $\mcA_n$, there are at least two singular atoms of $\mcA_{n+1}$ contained in $\alpha$.
\end{enumerate}
\end{lemma}

\begin{proof}
The construction proceeds as follows: first, we define $\tilde \Gamma_0$, which is the trivial group. Then we define an involution $\pi \in \Aut(\mcB_0)$ by setting $\pi(\alpha^1_0)=\beta^1_0$, $\pi(\beta^1_0)=\alpha^1_0$, and $\pi$ is the identity on the other atom. Our desired conditions are satisfied for $n=0$. 

Now assume that we are at step $n$ of our construction, i.e.~we have built $\tilde \Gamma_n$ and an automorphism $\pi$ of $\mcB_n$ satisfying our conditions. First, we extend $\tilde \Gamma_n$ to $\tilde \Gamma_{n+1}$ as in the proof of Lemma \ref{l:from_partitions_to_an_ample_group}, which is possible since $\mcA_{n+1}$ refines $\mcA_n$.
On atoms of $\mcB_{n+1}$ contained in an atom of $\mcA_n$ which is not singular, we have no choice for the extension of $\pi$: it must coincide with the extension of an element of $\tilde \Gamma_n$ to an automorphism of $\Aut(\mcB_{n+1})$, and that extension has already been defined. 

Let $\alpha$ be a singular atom of $\mcA_n$ and $\beta$ be an atom of $\mcA_{n+1}$ contained in $\alpha$. If the $\mcA_{n+1}$-orbit $O(\beta)$ is not exceptional, then it is $(\alpha,\pi(\alpha))$-balanced, so we may find an involution $\gamma \in \tilde \Gamma_{n+1}$ such that $\gamma(\delta) \subset \pi(\alpha)$ for all atoms $\delta$ of $O(\beta)$ contained in $\alpha$. Declare $\pi$ to be equal to $\gamma$ on those atoms.

So the real work consists of extending $\pi$ to the atoms of the exceptional $\mcA_{n+1}$-orbits contained in some singular atom of $\mcA_n$. Let $\tau_1,\ldots,\tau_p$ be singular atoms of $\mcA_n$ belonging to distinct $\pi$-orbits and such that for every singular atom $\alpha$ of $\mcA_n$ there exists $i$ such that $\alpha=\tau_i$ or $\alpha=\pi(\tau_i)$ (we use below the fact that $p \le h_n$, since there are fewer singular atoms than there are atoms in $\mcA_n$). Let $(O,O')$ be an exceptional pair of $\mcA_{n+1}$-orbits. Note that $(O,O')$ is a $(\tau_i,\pi(\tau_i))$-balanced pair of orbits for all $i$. 

Denote for all $i$ $m_O(i)= n_O(\tau_i)-n_O(\pi(\tau_i))$. We need to do something to balance the columns for which $|m_O(i)| \ge 1$; we distinguish two cases.

\begin{itemize}
    \item If $ m_O(i)=1$, choose two atoms $\theta_1(i)$, $\theta_2(i)$ of $O$ contained in $\tau_i$, and one atom $\theta_3(i)$ of $O$ contained in $\pi(\tau_i)$; as well as two atoms $\delta_1(i)$, $\delta_2(i)$ of $O'$ contained in $\pi(\tau_1)$, and an atom $\delta_3(i)$ of $O'$ contained in $\tau_i$. If $m_O(i)=-1$, do the same thing but with $\theta_1(i)$, $\theta_2(i)$ in $\pi(\tau_i)$ and $\theta_3(i)$ in $\tau_i$; and $\delta_1(i)$, $\delta_2(i)$ in $\tau_i$ while $\delta_3(i)$ belongs to $\pi(\tau_i)$.
    \item If $m_O(i) \ge 2$, we set aside atoms $\theta_1(i),\ldots,\theta_{m_O(i)}(i)$ of $O$ contained in $\tau_i$, and atoms $\delta_1(i),\ldots,\delta_{m_O(i)}(i)$ of $O'$ contained in $\pi(\tau_i)$; similarly, if $m_O(i)\le -2$ we set aside atoms $\theta_1(i),\ldots,\theta_{m_O(i)}(i)$ of $O$ contained in $\pi(\tau_i)$, and atoms $\delta_1(i),\ldots,\delta_{m_O(i)}(i)$ of $O'$ contained in $\tau_i$.
\end{itemize}

 Overall, this involves choosing fewer than $3 h_n N$ atoms in each of $O$, $O'$, so we can additionally ensure that no two of these atoms belong to the same $\tilde \Gamma_n$-orbit. We then set $\pi(\delta_j(i))= \theta_j(i)$, $\pi(\theta_j(i))= \delta_j(i)$ for all $i,j$.

We do this for all exceptional orbit pairs $(O,O')$ and all $i$. Now, in each exceptional orbit $O$ of $\mcA_{n+1}$, and for any $i$, there remain as many atoms of $O$ contained in $\tau_i$ on which $\pi$ is yet to be extended as there are such atoms contained in $\pi(\tau_i)$. This means that we can extend $\pi$ so that it coincides with an element of $\tilde \Gamma_{n+1}$ on those atoms.

This enables us to move on to the next step. Condition \ref{sparse} has been guaranteed when we chose our new singular atoms in distinct $\tilde \Gamma_n$-orbits. To see why condition \ref{no_isolated_points} holds, let $\alpha$ be a singular atom of $\mcA_n$. Since $\pi(\alpha) \not \sim_\Gamma \alpha$, there must be some exceptional column of $\mcA_{n+1}$ which is not $(\alpha,\pi(\alpha))$-balanced; and we took care to include more than two singular atoms contained in $\alpha$ in such a column (this is why we singled out the case $|m_O(i)|=1$ above).
\end{proof}

We let $\tilde \Gamma= \bigcup_n \tilde \Gamma_n$, $\Lambda= \bigcup_n \Lambda_n$. They are both ample groups, and $\Lambda$ is the full group generated by $\tilde \Gamma$ and $\pi$.

\begin{lemma}\label{l:orbits_of_Lambda_on_clopens}
For any $U,V \in \Clopen(X)$ such that $U \sim_\Gamma^* V$, there exists $\lambda \in \Lambda$ such that $\lambda(U)=V$.
\end{lemma}

\begin{proof}
First, note that for any clopen $U,V$ such that $U \sim_\Gamma V$, there exists some $n$ such that $(U,V)=(U_n,V_n)$, so that $(U,V)$ are $\mcA_n$-equivalent, hence also $\mcB_n$-equivalent, so there exists $\lambda \in \Lambda_n$ such that $\lambda(U)=V$. So $U \sim_\Lambda V$.

If $U \not \sim_\Gamma V$ but $U \sim_\Gamma^* V$, there exists $n$ such that $(U,V)=(U_n,V_n)$, so $(U, V)$ are $\mcB_n$-equivalent and $U \sim_\Lambda V$. 
\end{proof}

\begin{lemma}\label{l:Lambda_is_saturated}
$\Lambda$ acts minimally on $X$; $M(\Lambda)=M(\Gamma)$; and $\Lambda$ is a saturated ample group.
\end{lemma}

\begin{proof}
Given any nonempty clopen set $U$, there exists $\gamma_1,\ldots,\gamma_n \in \Gamma$ such that $X= \bigcup_{i=1}^n \gamma_i (U)$, since $\Gamma$ acts minimally. By Lemma \ref{l:orbits_of_Lambda_on_clopens}, for all $i$ there exists some $\lambda_i \in \Lambda$ such that $\lambda_i (U)= \gamma_i(U)$, whence $\Lambda$ also acts minimally. 

Let $G_\Gamma = \lset{g \in \Homeo(X) \colon \forall \mu \in M(\Gamma) \ g_*\mu=\mu}$. By construction, $\Lambda \subset G_\Gamma$ since each $\mcB_n$ is a $\sim_\Gamma^*$-partition (see Lemma \ref{l:from_partitions_to_an_ample_group}).

If $U,V \in \Clopen(X)$ are such that $\mu(U)=\mu(V)$ for all $\mu \in M(\Gamma)$, we know by the previous lemma that there exists $\lambda \in \Lambda$ such that $\lambda(U)=V$, so $\overline{\Lambda}$ contains $G_\Gamma$. Hence $\overline{\Lambda}= G_\Gamma$, thus for any Borel probability measure $\mu$ on $X$ we have
\begin{align*}
\mu \in M(\Lambda) & \Leftrightarrow \mu \in M(\overline{ \Lambda}) \\
                   & \Leftrightarrow \mu \in M(G_\Gamma) \\
                   & \Leftrightarrow  \mu \in M(\Gamma)  \quad \text{(recall that } M(\Gamma) = M(G_\Gamma) \text{ , see Lemma \ref{l:invariant_measures_stabilize})}     
\end{align*}

Thus $\overline{\Lambda}=G_\Gamma= \overline{[R_\Gamma]}$, and  $\overline{[R_\Gamma]}=\overline{[R_\Lambda]}$ since $M(\Lambda)=M(\Gamma)$ (see Lemma \ref{l:closure_full_group}). Hence $\overline{\Lambda}= \overline{[R_\Lambda]}$, i.e.~$\Lambda$ is saturated.
\end{proof}

\begin{lemma}
The orbits of the action of $\tilde \Gamma$ on $\Clopen(X)$ coincide with the orbits of the action of $\Gamma$ on $\Clopen(X)$.
\end{lemma}

\begin{proof}
One inclusion is immediate from the definition of $\tilde \Gamma$.

If $U,V$ are two clopens such that $U \sim_\Gamma V$, then there exists $n$ such that $(U,V)$ are $\mcA_n$-equivalent, so there exists $\gamma \in \tilde \Gamma_n$ such that $\gamma(U)=V$ (recall that the $\mcA_n$-orbit of any atom $\alpha$ of $\mcA_n$ coincides with $\tilde \Gamma_n \alpha$).
\end{proof}

Using Krieger's theorem, we conclude that $\Gamma$ and $\tilde \Gamma$ are conjugate.

\begin{defn}
We say that $x \in X$ is \emph{singular} if for any $n$ the atom $\alpha_n(x)$ of $\mcB_n$ which contains $x$ is singular  (recall that this means that $\pi$ does not coincide on $\alpha_n(x)$ with an element of $\tilde \Gamma_n$; and then $\pi(\alpha_n(x))$ does not belong to $\tilde \Gamma_n \alpha_n(x)$).
\end{defn}

If $x \in X$ is not singular, then $\pi$ coincides on a neighborhood of $x$ with an element of $\tilde \Gamma$. Since $\Lambda$ is generated, as a full group, by $\tilde \Gamma$ and $\pi$, this means that $R_\Lambda$ is the finest equivalence relation coarser than $R_{\tilde \Gamma}$ and such that $x $ is equivalent to $\pi(x)$ for each singular point $x$. Denote by $S$ the set of singular points. It is closed since it is an intersection of clopen sets.

\begin{lemma}
$S$ is $\tilde \Gamma$-sparse and has no isolated points.
\end{lemma}

\begin{proof}
Condition \ref{sparse} from the construction of $\pi$ implies that, for any $n$, there is no singular point besides $x$ in $\tilde \Gamma_n \alpha_{n+1}(x)$. In particular, $x$ is the unique singular point in $\tilde \Gamma_n x \cap S$ for any $n$. This proves that $S$ is $\tilde \Gamma$-sparse.

Let $x$ be singular, and $U$ a clopen subset containing $x$. There exists $n$ such that $\alpha_n(x) \subset U$; condition \ref{no_isolated_points} implies that there exists a singular atom of $\mcA_{n+1}$ contained in $\alpha_n(x)$ and disjoint from $\alpha_{n+1}(x)$. Since each singular atom contains a singular point, there is a singular point distinct from $x$ and belonging to $U$. So $S$ has no isolated points.
\end{proof}

Let us recap what we have done. Starting from a minimal ample group $\Gamma$, we built two new minimal ample groups $\tilde \Gamma$, $\Lambda$. The group $\tilde \Gamma$ is conjugated to $\Gamma$ since their actions on $\Clopen(X)$ have the same orbits; $\Lambda$ is saturated. The orbits of the action of $\Lambda$ on $X$ are obtained by gluing together pairs of orbits for the action of $\tilde \Gamma$ along a $\tilde \Gamma$-sparse closed subset which has no isolated points, so Theorem \ref{t:absorption} tells us that $\tilde \Gamma$ is orbit equivalent to $\Lambda$.

Hence $\Gamma$ is orbit equivalent to $\Lambda$; we have finally proved that every minimal ample group is orbit equivalent to a saturated minimal ample group, and this concludes the proof of the classification theorem for minimal ample groups.

\section{The classification theorem for minimal homeomorphisms}

The following theorem is a consequence of \cite{Giordano1995}*{Theorem 2.3}, as well as a particular case of \cite{Giordano2004}*{Theorem~4.16}; it can be seen as an absorption theorem where one only needs to glue two orbits together.
\begin{theorem}\label{t:classification_minimal_homeomorphisms}
Let $\varphi$ be a minimal homeomorphism, and $x \in X$. Then the relations induced by $\varphi$ and by $\Gamma_x(\varphi)$ are isomorphic.
\end{theorem}

\begin{proof}
Fix an element $y \in X$ which does not belong to the $\varphi$-orbit of $x$, and consider the ample group $\Delta= \Gamma_x(\varphi) \cap \Gamma_y(\varphi)$, which acts minimally (we skip the proof of this fact since we prove a more general statement below when establishing Theorem \ref{t:splitting_orbits}).

For any integer $N$, we can find a clopen $U$ such that $x,y \in U$ and $U \cap \varphi^i(U)=\emptyset$ for any $i \in \{1,\ldots,N\}$. By considering a Kakutani--Rokhlin partition with basis such a set $U$, we see thanks to Lemma \ref{l:original_Glasner_Weiss} that for any clopen $A$, $B$ such that $\mu(A)< \mu(B)$ for any $\mu \in M(\Gamma)$, there exists $\delta \in \Delta$ such that $\delta(A) \subset B$. 

As in the proof of Lemma \ref{l:closure_full_group}, it follows that 
$$\overline{[R_\Delta]}= \{g \in G \colon \forall \mu \in M(\varphi) \ g_* \mu=\mu\}= \overline{[R_{\Gamma_x(\varphi)}]}$$
Then the classification theorem for minimal ample groups implies that $\Delta$ is orbit equivalent to $\Gamma_x(\varphi)$.

Since $R_{\Gamma_x(\varphi)}$ is obtained from $R_\Delta$ by joining the $\Delta$-orbits of $y$ and $\varphi^{-1}(y)$ together, and $\Gamma_x(\varphi)$ is orbit equivalent to $\Delta$, we see that $R_\Delta$ is isomorphic to the relation obtained from $R_\Delta$ by joining the $R_\Delta$-orbits of $y,\varphi^{-1}(y)$ together.

Hence, there exist $z,t$ with distinct $\Gamma_x(\varphi)$-orbits such that $R_{\Gamma_x(\varphi)}$ is isomorphic to the relation obtained by joining the orbits of $z$ and $t$ together; and by Theorem \ref{t:Krieger_pointed}, there exists $g$ which realizes an isomorphism from $R_{\Gamma_x(\varphi)}$ to itself, with $g(z)=x$ and $g(t)=\varphi^{-1}(x)$.

Finally, $R_{\Gamma_x(\varphi)}$ is isomorphic to the relation induced from $R_{\Gamma_x(\varphi)}$ by joining the orbits of $x$ and $\varphi^{-1}(x)$ together, which is equal to $R_\varphi$.
\end{proof}

The classification theorem for minimal homeomorphisms follows immediately from this and the classification theorem for minimal ample groups.

\begin{proof}[Proof of the classification theorem for minimal homeomorphisms]
Assume that $\varphi$, $\psi$ are two minimal homeomorphisms of $X$ such that $M(\varphi)= M(\psi)$. Recall that for any $x$ we have $M(\Gamma_x(\varphi))=M(\varphi)$ and $M(\Gamma_x(\psi))=M(\psi)$.

Thus if $M(\varphi)=M(\psi)$ then $\Gamma_x(\varphi)$ and $\Gamma_x(\psi)$ are orbit equivalent by the classification theorem for minimal ample groups, and then by the previous theorem $\varphi$ and $\psi$ are orbit equivalent.
\end{proof}

Using the same approach, we can also recover Theorem~4.16 of \cite{Giordano2004}. We first note the following easy fact (we have used earlier in this paper an analogue of this for minimal ample groups).

\begin{lemma}
Assume that $\varphi$ is a minimal homeomorphism of $X$, and that $Y$ is a closed subset of $X$ which meets every $\varphi$-orbit in at most one point. 
Then for any $N$ one can find a clopen subset $U$ containing $Y$ and such that $U \cap\varphi^i(U) = \emptyset$ for all $i \in \{1,\ldots,N\}$. 
\end{lemma}

\begin{proof}
Fix $N$. Since $Y,\varphi(Y),\ldots,\varphi^N(Y)$ are closed and pairwise disjoint, there exist disjoint clopen sets $U_i$ such that $\varphi^i(Y) \subset U_i$ for all $i \in \{0,\ldots,N\}$. Set
$$U=\bigcap_{i=0}^N \varphi^{-i}(U_i) $$
Then $U$ is clopen, contains $Y$, and $U \cap \varphi^{i}(U) \subseteq U_0 \cap U_i= \emptyset$ for all $i \in \{1,\ldots,N\}$.
\end{proof}

\begin{nota}[see \cite{Giordano2004}*{Theorem~4.6}]
Let $\varphi$ be a minimal homeomorphism, and $Y$ be a closed set meeting each $\varphi$-orbit in at most one point. For $y \in Y$, recall that 
\[O^+(y)= \{\varphi^n(y) \colon n \ge 0\} \quad \text{ and } \quad O^-(y)= \{\varphi^n(y) \colon n < 0\}\]

We denote by $R_Y$ the equivalence relation obtained from $R_\varphi$ by splitting the $\varphi$-orbit of each $y \in Y$ into $O^+(y)$ and $O^-(y)$, and leaving the other $R$-classes unchanged.
\end{nota}

\begin{theorem}[\cite{Giordano2004}*{Theorem~4.16}]\label{t:splitting_orbits}
Let $\varphi$ be a minimal homeomorphism, and $Y$ be a closed set which meets every $\varphi$-orbit in at most one point. Then $R_Y$ and $R_\varphi$ are isomorphic.
\end{theorem}

\begin{proof}
Consider the group $\Gamma_Y= \bigcap_{y \in Y} \Gamma_y(\varphi)$. It is ample since it is an intersection of ample groups. 

We saw that for any given $N$ there exists a Kakutani--Rokhlin partition for $\varphi$ with base a clopen set $U$ containing $Y$ and such that $U \cap \varphi^i(U)= \emptyset$ for all $i \in \{1,\ldots,N\}$.
Fix $x$ whose orbit does not intersect $Y$, and find such a partition with $x$ not belonging to $U$ or $\varphi^{-1}(U)$ (just shrink $U$ if necessary) and say $N=2$. Looking at the atom which contains $x$, we see that $\varphi(x)$ and $\varphi^{-1}(x)$ are both $\Gamma_Y$-equivalent to $x$ (since the restriction of $\varphi^{\pm 1}$ to some neighborhood of $x$ belongs to $\Gamma_Y$). This proves that the $\Gamma_Y$-orbit of any element whose $\varphi$-orbit does not intersect $Y$ coincides with its $\varphi$-orbit: we can always move via $\varphi^{\pm1}$ along this orbit. 

By definition of $\Gamma_Y$, the orbit of each $y \in Y$ splits in at least two $\Gamma_Y$-orbits; and for each $n \ge 0$ and $y \in Y$, the restriction of $\varphi^n$ to some neighborhood of $y$ belongs to $\Gamma_Y$ (take a partition as above with $N = n$); same argument for negative semi-orbits. Thus $R_Y$ is the relation induced by $\Gamma_Y$, and in particular $\Gamma_Y$ acts minimally.

As in the proof of Theorem \ref{t:classification_minimal_homeomorphisms}, we deduce from the existence of Kakutani--Rokhlin partitions with arbitrarily large height and basis containing $Y$ that $\overline{[R_Y]}=\overline{[R_\varphi]}$, so by the classification theorem $R_Y$ is isomorphic to $R_\varphi$.
\end{proof}

We conclude this paper by improving our absorption theorem \ref{t:absorption} (the result below is still a particular case of the absorption theorems in \cite{Giordano2004}, \cite{Giordano2008}, \cite{Matui2008}). 

\begin{theorem}\label{t:absorption2}
Let $\Gamma$ be an ample subgroup acting minimally. Let $K = K_1 \sqcup \sigma(K_1)$, where $K_1$ is closed and $\sigma \in \Homeo(K)$ is an involution. Assume that $K$ is $\Gamma$-sparse.

Denote by $R_{\Gamma,K}$ the finest equivalence relation coarser than $R_\Gamma$ and for which $k,\sigma(k)$ are equivalent for all $k \in K$. Then there exists $f \in \Homeo(X)$ such that $f \Gamma f^{-1}$ induces $R_{\Gamma,K}$ (hence $R_{\Gamma,K}$ and $R_\Gamma$ are isomorphic).
\end{theorem}

\begin{proof}
By the constructions used in the previous section, we know that we can find a nonempty closed subset $F$ of $X$ without isolated points, and a homeomorphism $\pi$ of $F$, such that $x \ne \pi(x)$ for all $x \in F$ and the intersection of any $\Gamma$-orbit with $F$ is either empty or equal to $\{x,\pi(x)\}$ for some $x \in F$. Note that there exists a homeomorphic embedding $g \colon K \to F$ such that $g(\sigma(x))=\pi(g(x))$ for all $x \in K$. Recall that $K=K_1 \sqcup \sigma(K_1)$; let $Y=g(K_1)$. 

Let $\varphi$ be a minimal homeomorphism of $X$ and $x \in X$ such that $\Gamma = \Gamma_x(\varphi)$ (see Theorem \ref{t:ample_vs_integers}). We can assume that $x$ is not in the $\varphi$-orbit of any element of $Y$.
As in the proof of Theorem \ref{t:splitting_orbits}, denote $\Gamma_Y= \bigcap_{y \in Y} \Gamma_y(\varphi)$. We know that $\Gamma_Y$ is orbit equivalent to $\Gamma_x(\varphi)=\Gamma$, so there exists $\Lambda$ conjugate to $\Gamma$ in $\Homeo(X)$ and which induces $R_{\Gamma_Y}$. 

Further, $R_{\Gamma}$ is obtained from $R_{\Gamma_Y}=R_\Lambda$ by joining together the $\Lambda$-orbits of $y$ and $\varphi^{-1}(y)$ for all $y \in Y$. 

Consider the homeomorphism $h \colon Y \sqcup \varphi^{-1}(Y) \to K$ defined by $h(y)=g^{-1}(x)$, $h(\varphi^{-1}(y))=\sigma(g^{-1}(x))$  for all $y \in Y$. Applying theorem \ref{t:Krieger_pointed} to $\Lambda$, $\Gamma$ and $h$, we obtain a homeomorphism $f$ of $X$ such that $f \Lambda f^{-1}= \Gamma$ and $f_{|Y \sqcup \varphi^{-1}(Y)}=h$. 

Then $f \Gamma f^{-1}$ induces the relation obtained from $R_\Gamma$ by gluing together the $\Gamma$-orbit of $h(y)$ and $h(\varphi^{-1}(y))= \sigma(h(y))$ for all $y \in Y$. This relation is exactly $R_{\Gamma,K}$.
\end{proof}

\bibliography{biblio_GPS}

\end{document}